\documentclass[twoside]{article}
\usepackage{enumerate,mathrsfs,amsfonts,amsxtra,latexsym,amssymb,amsthm,verbatim,amsmath,graphicx}
\usepackage{mathpazo}
\usepackage[dvipdfm, colorlinks, linkcolor=green, anchorcolor=blue, citecolor=red]{hyperref}

\catcode`@=12

\newtheorem{theorem}{Theorem}[section]
\newtheorem{definition}[theorem]{Definition}
\newtheorem{lemma}[theorem]{Lemma}

\newtheorem{proposition}[theorem]{Proposition}

\begin{document}
\abovedisplayskip=6pt plus 1pt minus 1pt \belowdisplayskip=6pt
plus 1pt minus 1pt
%-------------------  First Head  -----------------------------------------
\thispagestyle{empty} \vspace*{-1.0truecm} \noindent
\vskip 10mm

\begin{center}{\large A polynomial invariant of virtual links\footnotemark\\[2mm]
\footnotetext{\footnotesize The authors are supported by NSFC (No. 11171025) and Science Foundation for The Youth Scholars of Beijing Normal University}} \end{center}

\vskip 5mm
\begin{center}{Zhiyun Cheng\quad Hongzhu Gao\\
{\small School of Mathematical Sciences, Beijing Normal University
\\Laboratory of Mathematics and Complex Systems, Ministry of
Education, Beijing 100875, China
\\(email: czy@bnu.edu.cn; hzgao@bnu.edu.cn)}}\end{center}

\vskip 1mm

\noindent{\small {\small\bf Abstract} In this paper, we define some polynomial invariants for virtual knots and links. In the first part we use Manturov's parity axioms \cite{Man2010} to obtain a new polynomial invariant of virtual knots. This invariant can be regarded as a generalization of the odd writhe polynomial defined by the first author in \cite{Che2012}. The relation between this new polynomial invariant and the affine index polynomial \cite{Kau2012,Fol2012} is discussed. In the second part we introduce a polynomial invariant for long flat virtual knots. In the third part we define a polynomial invariant for 2-component virtual links. This polynomial invariant can be regarded as a generalization of the linking number.
\ \

\vspace{1mm}\baselineskip 12pt


\noindent{\small\bf MR(2010) Subject Classification} 57M25 57M27\ \ {\rm }}

\vskip 1mm

\vspace{1mm}\baselineskip 12pt

\section{Introduction}
This paper concerns some polynomial invariants of virtual knots and virtual links. Virtual knot theory can be regarded as a generalization of the classical knot theory. With this viewpoint some classical polynomial invariants can be defined similarly on virtual knots, such as the Jones polynomial \cite{Kau1999}. Besides of these classical polynomial invariants, some generalizations, such as the arrow polynomial \cite{Dye2009} and the Miyazawa polynomial \cite{Miy2008} (which are equivalent) were introduced recently.

The idea of parity plays an important role in virtual knot theory. For example, the writhe of a virtual knot diagram is not an invariant. However by coloring each real crossing point odd or even properly, L. Kauffman has shown that the odd writhe \cite{Kau2004} is a simple and enlightening virtual knot invariant. In \cite{Man2010} V. O. Manturov generalised the idea of parity into parity axioms (see Section 3), which told us what kind of parity will be useful. Inspired by the definition of warping polynomial \cite{Shim2010,Shim2011}, the first author defined the odd writhe polynomial in \cite{Che2012}, which can be regarded as a generalization of the odd writhe. The main idea of the odd writhe polynomial is assigning a weight to each crossing, the similar idea was used by A. Henrich to define three kinds of degree one Vassiliev invariants in \cite{Hen2010}. In order to define the odd writhe polynomial first we label each real crossing point $c_i$ with an integer number $N(c_i)$ according to some rules, then we take the sum of $w(c_i)t^{N(c_i)}$ for all odd crossings. Here $w(c_i)$ denotes the writhe $($or the sign$)$ of $c_i$, see the figure below.
\begin{center}
\includegraphics{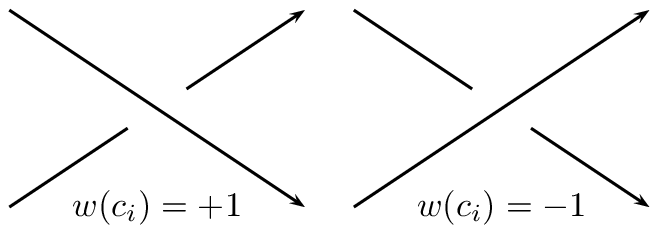}\\
Figure 1
\end{center}

One disadvantage of the odd writhe polynomial is that it vanishes if the virtual knot diagram contains no odd crossings. If all the classical crossing points of a virtual knot diagram are even then we say this diagram is \emph{even}. In order to investigate even knot diagrams, we consider a sequence of parities which were introduced in \cite{Man2010}. With these parities we introduce a sequence of polynomial invariants $f_0(t), f_1(t), f_2(t), \cdots$, then we define the \emph{writhe polynomial} $W_K(t)$ to be
\begin{center}
$W_K(t)=\sum\limits_{i=0}^{\infty}f_i(t)$.
\end{center}
We shall prove that this polynomial is a virtual knot invariant. Similar to the odd writhe polynomial, the writhe polynomial can be used to detect some virtual knots from its inverse and mirror image. In fact with another viewpoint of the writhe polynomial, we can easily find that the writhe polynomial is essentially equivalent to the affine index polynomial introduced by L. Kauffman very recently \cite{Kau2012,Fol2012}. More precisely, the difference between the writhe polynomial and the affine index polynomial (with a shift) is a simple virtual knot invariant.

In the second part of this paper we consider the long flat virtual knots. A long flat virtual knot can be regarded as  a 1-1 tangle. A diagram of long flat virtual knot contains two kinds of crossings, the flat crossing and the virtual crossing. By replacing each flat crossing with an overcrossing or an undercrossing we get a long virtual knot. If a flat long virtual knot is not trivial, it follows that all long virtual knots obtained from this flat long virtual knot are nontrivial. Usually it is difficult to detect whether a long flat virtual knot is trivial or not. Our polynomial invariant of long flat virtual knots may help doing it. For the closure of a long flat virtual knot, i.e. a flat virtual knot, we remark that some polynomial invariants can be defined similarly, see \cite{Hen2010} and \cite{IKL2012}.

In the third part of this paper we introduce a polynomial invariant which can be regarded as a generalization of the linking number. The writhe polynomial and the affine index polynomial both concern virtual knot diagrams. Obviously for a link diagram one can define a family of polynomial invariants by considering each component respectively. But these invariants only reflect the knotted information of each component, the linking information between different components are lost. By investigating the Gauss diagram of a virtual link diagram $K_1\cup K_2$ we introduce two polynomials $F(t)$ and $G(t)$. Although both $F(t)$ and $G(t)$ are not invariants, the product of them is a virtual link invariant. We name it the \emph{linking polynomial}
\begin{center}
$L_{K_1\cup K_2}(t)=F(t)G(t)$
\end{center}
of a given virtual link diagram $K_1\cup K_2$. Some examples of virtual links with trivial linking number but non-trivial linking polynomial will be given. We will show that for classical links the linking polynomial contains the same information as the linking number, this allows us to detect whether a virtual link diagram is virtually linked or not.

Our paper is organized as follows. In section 2 we give a short introduction to virtual knot theory. Some simple invariants, such as the odd writhe invariant and the odd writhe polynomial will be reviewed. In section 3  we present the parity axioms proposed by V. O. Manturov. With a sequence of parities we define the writhe polynomial and show that it is a virtual knot invariant. The relation between the writhe polynomial and Kauffman's affine index polynomial is also discussed. In section 4 we will show how to define the writhe polynomial on a long flat virtual knot. In the last section we construct a polynomial invariant of 2-component virtual links. We name it the linking polynomial, some interesting examples will be given, and the relation between linking polynomial and linking number will be discussed.

\section{Virtual knot theory and the odd writhe polynomial}
In the beginning of this section we will take a short review of virtual knot theory. For more details we recommend \cite{Kau1999} in which virtual knot was proposed for the first time, and \cite{Kau2011} for a good survey.

From the diagrammatic viewpoint, classical knot theory can be described as follows. Given a 4-valent planar graph, replace each vertex by an overcrossing or an undercrossing, we can obtain a link diagram. A pair of link diagrams are equivalent if and only if one diagram can be obtained from the other by a sequence of Reidemeister moves. In virtual knot theory, besides overcrossing and undercrossing we have another crossing type, the virtual crossing. A virtual crossing point is usually represented by a small circle placed around the crossing point. Instead of the classical three Reidemeister moves, we need to generalize the Reidemeister moves as below:
\begin{center}
\includegraphics{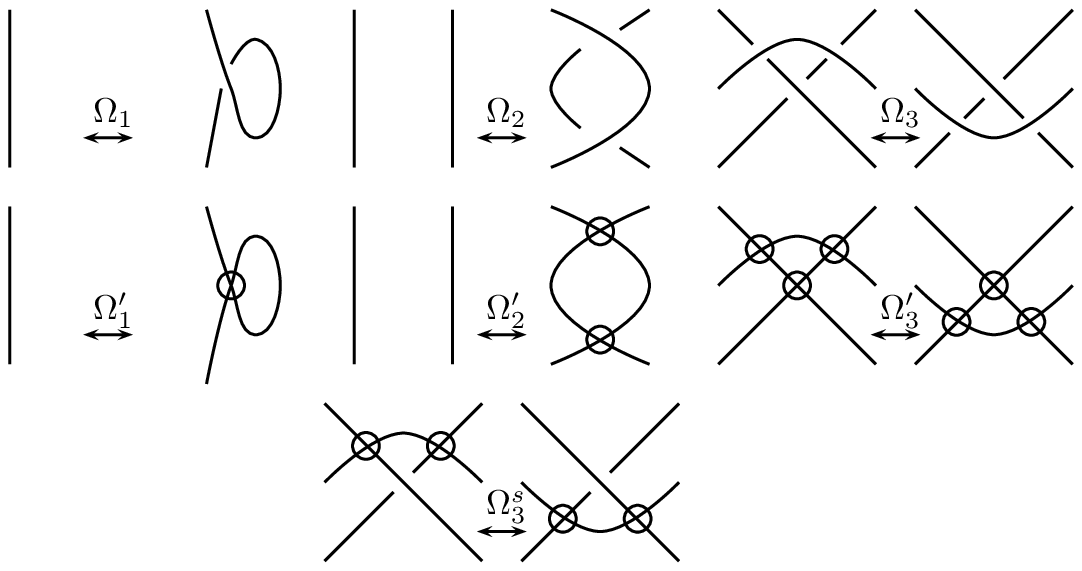}\\
Figure 2
\end{center}
It is easy to find that $\Omega_1, \Omega_2, \Omega_3$ are exactly those classical Reidemeister moves, $\Omega'_1, \Omega'_2, \Omega'_3$ can be regarded as the virtual versions of $\Omega_1, \Omega_2, \Omega_3$. The last one $\Omega_3^s$ is a version of the third Reidemeister move. For simplicity one can summarize $\Omega'_2, \Omega'_3$ and $\Omega_3^s$ by a \emph{detour move}, see the figure below.
\begin{center}
\includegraphics[width=8cm]{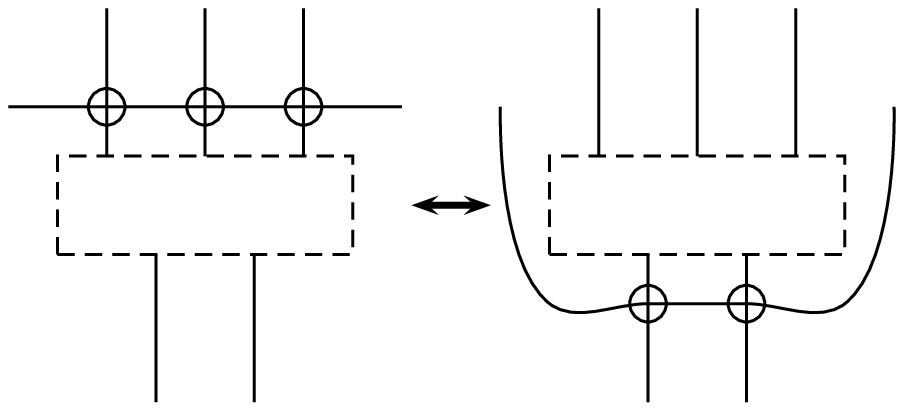}\\
Figure 3
\end{center}
If two virtual link diagrams can be connected by a finite sequence of these generalized Reidemeister moves then we say they are \emph{equivalent}.

There are mainly two motivations for introducing virtual knot theory. The first motivation is the interpretation of virtual links as stably embedded circles in thickened surfaces. The classical knot theory studies the embeddings of $S^1$'s in $S^3$ $($or $S^2\times I)$ up to isotopy. It is a natural question to consider the embeddings of $S^1$'s in $\Sigma_g$, here $\Sigma_g$ denotes the surface with genus $g$. For each virtual link diagram there is a corresponding embedding of $S^1$'s in $\Sigma_h$ by adding some 1-handles on the original $S^2$, here $h$ denotes the number of virtual crossings. We say two such embeddings are \emph{stably equivalent} if one can be obtained from the other by an ambient isotopy in the thickened surface, homeomorphisms of the surfaces, and the addition or subtraction of handles which does not intersect the image of the link. Then we have
\begin{theorem}[\cite{Kau1999,Kau2005}]
Two virtual links are equivalent if and only if their corresponding embeddings are stably equivalent.
\end{theorem}
From this topological viewpoint some invariants can be defined naturally. For instance the minimal genus of the surface mentioned above is obviously an invariant of virtual links. By studying this invariant it was proved that the connected sum of two non-trivial virtual knots is also non-trivial \cite{Man2008}. In \cite{Kup2003} Kuperberg showed that for a virtual knot the embedding type of the minimal genus surface is unique.

The second motivation of proposing virtual knot theory comes from the realization of Gauss diagrams. Given a classical knot diagram $K$, the \emph{Gauss diagram} of $K$ is an oriented circle where the preimages of each crossing points are indicated. The two preimages of one crossing are connected by an arrow, directed from the preimage of the overcrossing to the preimage of the undercrossing. Then we assign a sign to each chord according to the writhe of the corresponding crossing point. The figure below gives an example of the trefoil knot.
\begin{center}
\includegraphics[width=6cm]{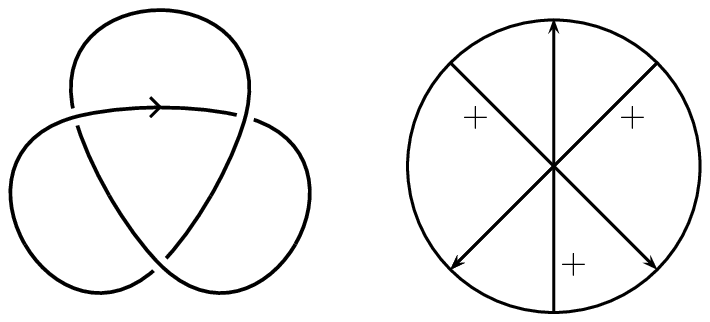}\\
Figure 4
\end{center}

Note that if a Gauss diagram can be realized by a classical knot diagram then the corresponding classical knot diagram is unique. However not every Gauss diagram can be realized by a classical knot diagram. In order to realize all Gauss diagrams one has to add some virtual crossings on the knot diagram. The figure below shows a virtual trefoil knot and its Gauss diagram.
\begin{center}
\includegraphics{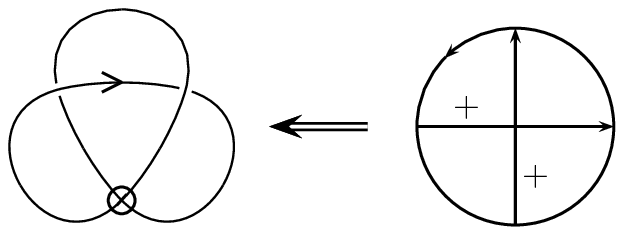}\\
Figure 5
\end{center}
Hence given a Gauss diagram, we can always construct a virtual knot diagram representing it. It is evident that this virtual knot diagram is not unique. But it is not difficult to observe that $\Omega'_1, \Omega'_2, \Omega'_3$ and $\Omega_3^s$ all preserve the Gauss diagram. In fact we have the following theorem.
\begin{theorem}[\cite{Gou2000}]
A Gauss diagram uniquely defines a virtual knot isotopy class.
\end{theorem}

Now we take a brief review of the odd writhe and the odd writhe polynomial introduced in \cite{Kau2004} and \cite{Che2012} respectively. Given a virtual knot diagram there is an associated Gauss diagram. For an arrow in the Gauss diagram, if there are odd number of vertices on both sides of it, then we name this arrow an \emph{odd} arrow. Otherwise we say it is \emph{even}. The parity of each crossing point on the diagram can be defined to be the parity of the corresponding arrow. We remark that an arrow in a Gauss diagram is odd if and only if there are odd number of arrows which cross it transversely. Following \cite{Kau2004} we use $Odd(K)$ to denote all the odd crossings of $K$. It is easy to find that $Odd(K)=\emptyset$ if $K$ is a classical knot diagram. Then the \emph{odd writhe} of a virtual knot $K$ can be defined as
\begin{center}
$J(K)=\sum\limits_{c_i\in Odd(K)}w(c_i)$,
\end{center}
here $c_i$ and $w(c_i)$ denote a crossing point and its writhe respectively. The odd writhe is a simple invariant of virtual knots. Since classical knots have no odd crossings, if a virtual knot has non-trivial odd writhe then it is non-classical and hence non-trivial.

For a given virtual knot diagram $K$, let us consider its associated Gauss diagram. If $K$ contains $n$ classical crossing points, then there are $n$ arrows in the Gauss diagram and the $2n$ endpoints of these arrows divide the circle into $2n$ arcs. To define the odd writhe polynomial, first we assign an integer to each arc as follows. Choose an arc and a point on it, travel along the circle according to the orientation (without loss of generality we assume the orientation of the circle is anti-clockwise). Then there are two types of arrows, for the first type we will meet the preimage of an undercrossing earlier than the preimage of an overcrossing, otherwise we say the it belongs to the second type. Then we assign the sum of the writhe of the first type crossings to the chosen arc. Repeating this process until each arc has an assigned integer. According to the assignment, the assigned integers of two adjacent arcs are different by one, the figure below lists all the possibilities.
\begin{center}
\includegraphics{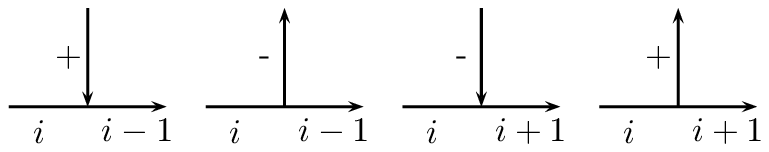}\\
Figure 6
\end{center}

Next we assign an integer to each arrow. Consider an arrow $c_i$ in the Gauss diagram. According to the writhe of $c_i$, the assigned integers on the two sides of each end point of $c_i$ can be described as follows:
\begin{center}
\includegraphics{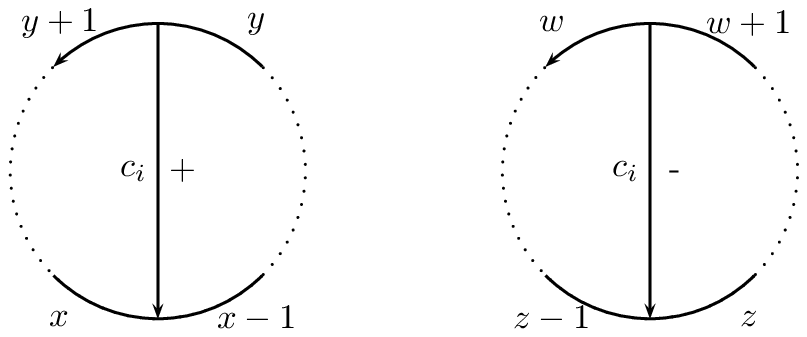}\\
Figure 7
\end{center}
Then we assign an integer $N(c_i)$ to each arrow $c_i$ by
\begin{center}
$N(c_i)=
\begin{cases}
x-y,& \text{if $w(c_i)=+1$;}\\
z-w,& \text{if $w(c_i)=-1$.}
\end{cases}$
\end{center}
Now the \emph{odd writhe polynomial} of $K$ can be defined as
\begin{center}
$f_K(t)=\sum\limits_{c_i\in Odd(K)}w(c_i)t^{N(c_i)}.$
\end{center}

The odd writhe polynomial is a virtual knot invariant and it is evident to see that $J(K)=f_K(\pm1)$ $($here we use the fact that $N(c_i)$ is an even integer if $c_i$ is odd$)$. Here we list some simple properties of the odd writhe polynomial.
\begin{proposition}[\cite{Che2012}]
The odd writhe polynomial satisfies
\begin{enumerate}
  \item $f_{\overline{K}}(t)=f_K(t^{-1})\cdot t^2$, here $\overline{K}$ denotes the inverse of $K$.
  \item $f_{K^*}(t)=-f_K(t^{-1})\cdot t^2$, here $K^*$ denotes the mirror image of $K$.
  \item $f_{K_a\#K_b}(t)=f_{K_a}(t)+f_{K_b}(t)$, here $K_a\#K_b$ denotes one connected sum of $K_a, K_b$ with an arbitrarily chosen of the connection place.
\end{enumerate}
\end{proposition}

\section{A generalization of the odd writhe polynomial}
As we have seen in Section 2, the odd writhe polynomial is an invariant of virtual knots which can be used to distinguish some virtual knots from its inverse and mirror image. Unfortunately there is a visible disadvantage to the odd writhe polynomial, i.e. if a virtual knot diagram contains no odd crossings, then the odd writhe polynomial is trivial. We use \emph{even diagram} to denote this kind of virtual knot diagrams. For example all classical knot diagrams are even. Now we will generalize the odd writhe polynomial to the \emph{writhe polynomial}, which need not to be trivial on even diagrams. Later we will discuss the relation between the writhe polynomial and the affine index polynomial defined in \cite{Kau2012} and \cite{Fol2012}.

Before giving the definition of the writhe polynomial, we give a review of the parity axioms introduced in \cite{Man2010}. Note that the parity axioms defined in \cite{Man2010} are valid for several knot theories (for example the free knot theory etc.), but now we are discussing virtual knots hence the statement here is a little different from that in \cite{Man2010}. See \cite{Man2012} for a general definition of parity.
\begin{definition}
Let $K_1$ and $K_2$ be a pair of virtual knot diagrams such that they can be obtained from each other by a single Reidemeister move. We have a rule which assigns each real crossing point the number 0 (even) or 1 (odd). We say the rule satisfies the parity axioms if it satisfies the following conditions.
\begin{enumerate}
  \item If $K_2$ is obtained from $K_1$ by the first Reidemeister move, then the crossing involved in the Reidemeister move is even.
  \item  If $K_2$ is obtained from $K_1$ by the second Reidemeister move, then the two crossing points involved in the Reidemeister move have the same parity.
  \item If $K_2$ is obtained from $K_1$ by the third Reidemeister move, then
\begin{enumerate}
  \item the corresponding crossing have the same parity
  \item in the three crossings $\{a, b, c\}$, the number of odd crossings is 0 or 2.
\end{enumerate}
\begin{center}
\includegraphics{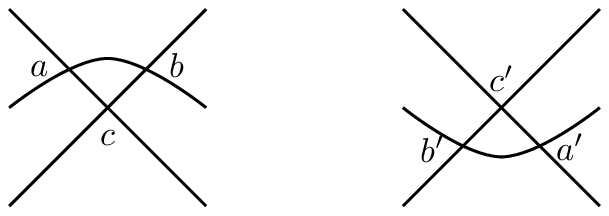}\\
Figure 8
\end{center}
  \item For each Reidemeister move, the parity of a crossing point which is not involved in the Reidemeister move does not change.
\end{enumerate}
\end{definition}

It is easy to observe that the parity introduced in the definition of the odd writhe satisfies the parity axioms above. One can easily conclude that the odd writhe is a virtual knot invariant from these conditions. In order to generalize the odd writhe polynomial to even diagrams, we define an index for each arrow in the Gauss diagram.

Given a Gauss diagram, consider an arrow $c$ of it. Let $r_+$ $(r_-)$ denotes the number of positive (negative) arrows crossing $c$ from left to right, let $l_+$ $(l_-)$ denotes the number of positive (negative) arrows crossing $c$ from right to left. See the figure below.
\begin{center}
\includegraphics{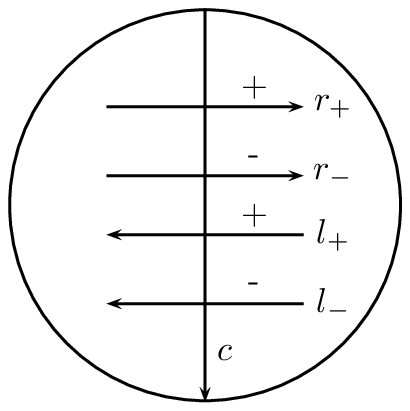}\\
Figure 9
\end{center}
Then the \emph{index} of $c$ $($a similar \emph{intersection index} was introduced by A. Henrich in \cite{Hen2010}$)$ can be defined as
\begin{center}
Ind$(c)=r_+-r_--l_++l_-$.
\end{center}
Given an arrow $c$ in a Gauss diagram, the relation between Ind$(c)$ and $N(c)$ is very simple.
\begin{lemma}
$N(c_i)=$Ind $(c_i)+1$.
\end{lemma}
\begin{proof}
According to Figure 9, it is obvious that each arrow crossing $c$ from left to right with sign $\pm1$ will contribute $\pm1$ to $N(c)$, and each arrow crossing $c$ from right to left with sign $\pm1$ will contribute $\mp1$ to $N(c)$. Besides, the arrow $c$ itself will contribute 1 to $N(c)$. Therefore the result follows.
\end{proof}

We remark that the parity proposed in the definition of the odd writhe can be restated as below
\begin{center}
$c$ is $\begin{cases}
\text{odd}, & \text{if \; Ind}(c)=1 \quad \text{mod 2}\\
\text{even}, & \text{if \; Ind}(c)=0 \quad \text{mod 2}
\end{cases}$
\end{center}

Now let us consider an even knot diagram, i.e. each crossing point is even. With the parity above there is no odd arrows in the Gauss diagram, hence no information can be obtained. For this reason, we choose another parity on each arrows, which was also given in \cite{Man2010}.
\begin{center}
$c$ is $\begin{cases}
\text{odd}, & \text{if \; Ind}(c)=2 \quad \text{mod 4}\\
\text{even}, & \text{if \; Ind}(c)=0 \quad \text{mod 4}
\end{cases}$
\end{center}
Similar to the odd writhe polynomial, we can define a polynomial $f_1(t)$ with this new parity.
\begin{center}
$f_1(t)=\sum\limits_{c_i\in Odd_1(K)}w(c_i)t^{N(c_i)}.$
\end{center}
Here $Odd_1(K)$ denotes the set of all the odd arrows with the new parity defined above, $w(c_i)$ and $N(c_i)$ are the same as before.

Obviously if each crossing $c_i$ of a virtual knot diagram satisfies
\begin{center}
Ind$(c_i)=0$ \quad mod 4,
\end{center}
then $f_1(t)=0$. In general we can define a sequence of parities as follows \cite{Man2010}.
\begin{center}
$c$ is $\begin{cases}
\text{odd}, & \text{if \; Ind}(c)=2^k \quad \text{mod } 2^{k+1}\\
\text{even}, & \text{if \; Ind}(c)=0 \quad \text{\;\,mod } 2^{k+1}
\end{cases}$
\end{center}
Let $Odd_k(K)$ denotes the set of the odd crossing points of $K$ according to this parity. It is evident that for a given virtual knot diagram there exists some integer $N\geq 0$ such that $Odd_k(K)=\emptyset$ if $k>N$. This is because the number of crossing points is finite. Now we can define a sequence of polynomials as below
\begin{center}
$f_k(t)=\sum\limits_{c_i\in Odd_k(K)}w(c_i)t^{N(c_i)}.$
\end{center}
Note that $f_0(t)$ is exactly the odd writhe polynomial. Now we can define the \emph{writhe polynomial} of a virtual knot $K$ to be
\begin{center}
$W_K(t)=\sum\limits_{k=0}^\infty f_k(t)$.
\end{center}
As we have mentioned above there are only finitely many nonempty sets $Odd_k(K)$, hence the writhe polynomial is actually a finite sum. For the writhe polynomial we have the following theorem.
\begin{theorem}
For each $k\geq 0$, the polynomial $f_k(t)$ is a virtual knot invariant. As a corollary the writhe polynomial is also a virtual knot invariant.
\end{theorem}
\begin{proof}
If $k=0$, the invariance of the odd writhe polynomial has been proved in \cite{Che2012}. If $k>0$, the result mainly follows from the fact that each corresponding parity satisfies the parity axioms.

In order to show the invariance of $f_k(t)$, it is sufficient to prove that $f_k(t)$ is invariant under Reidemeister moves. According to \cite{Pol2010}, the following Reidemeister moves $\{\Omega_{1a}, \Omega_{1b}, \Omega_{2a}, \Omega_{3a}\}$ is a generating set of all oriented Reidemeister moves. Hence it suffices to show that $f_k(t)$ is kept under these four kinds of Reidemeister moves.
\begin{center}
\includegraphics{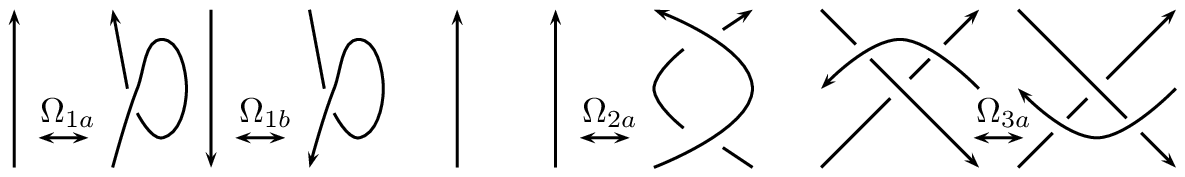}\\
Figure 10
\end{center}

\begin{enumerate}
\item The invariance under $\Omega_{1a}$ and $\Omega_{1b}$.
\begin{center}
\includegraphics{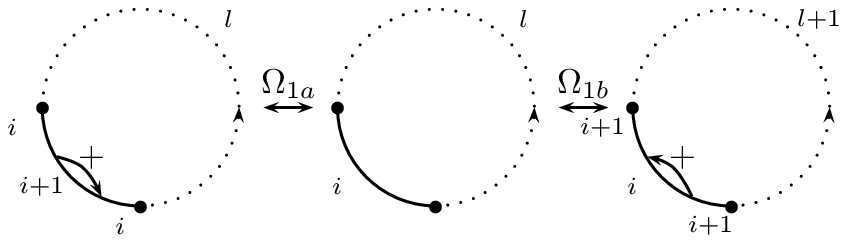}\\
Figure 11
\end{center}
The figure above indicates the corresponding move of $\Omega_{1a}$ and $\Omega_{1b}$ on the Gauss diagram. Consider the $\Omega_{1a}$-move, notice that the assigned number and indices of all other arrows are invariant. And the isolated arrow has index zero, hence the polynomial $f_k(t)$ is kept under $\Omega_{1a}$. For $\Omega_{1b}$, the assigned number on all other arcs are increased by one, hence by definition the assigned number and indices of arrows that are not involved in $\Omega_{1b}$ are also invariant. Since the index of the isolated arrow is zero, it has no contribution to $f_k(t)$. It means that $f_k(t)$ is not changed under $\Omega_{1b}$.
\item The invariance under $\Omega_{2a}$.
\begin{center}
\includegraphics{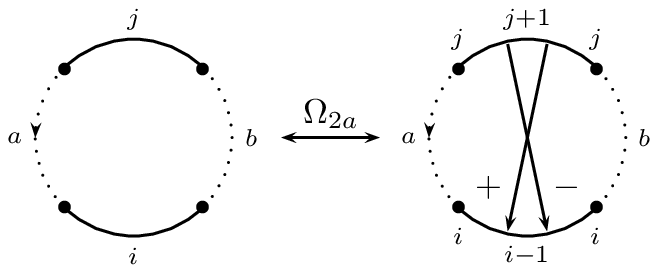}\\
Figure 12
\end{center}
Let us consider the behavior of $\Omega_{2a}$ on the Gauss diagram. Let $c_1$ and $c_2$ denote the pair of arrows shown above. From the figure it is easy to find that the assigned number of all other arrows are invariant. Since $c_1$ and $c_2$ have different signs, it follows that the indices of other arrows are also preserved. Now let us consider $c_1$ and $c_2$, according to the definition of the index it is not difficult to observe that they have the same indices. Besides we also have $N(c_1)=N(c_2)=i-j$. However the signs of them are opposite, hence the contribution of them to $f_k(t)$ is $t^{i-j}-t^{i-j}$ if Ind$(c_1)=$ Ind$(c_2)=2^k$ mod $2^{k+1}$, and 0 otherwise. In conclusion we find that $f_k(t)$ is invariant under $\Omega_{2a}$.

\item The invariance under $\Omega_{3a}$.
\begin{center}
\includegraphics{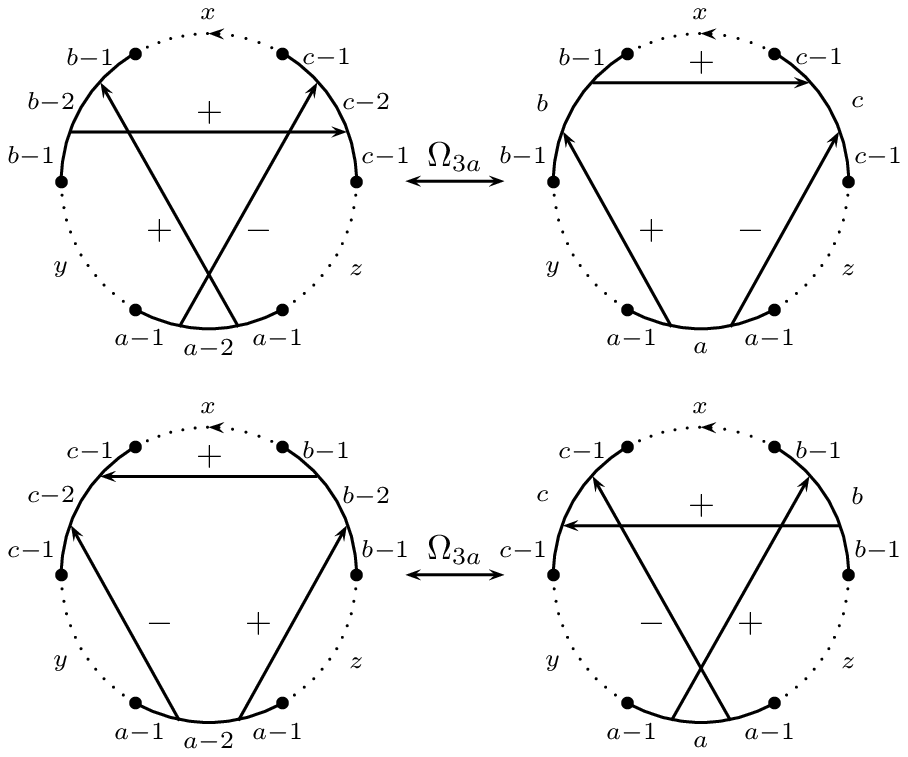}\\
Figure 13
\end{center}
Now let us investigate the behavior of $\Omega_{3a}$ on the Gauss diagram. According to the different connecting methods outside the local diagram of $\Omega_{3a}$, there are two kinds of corresponding Gauss diagrams as above. As before if an arrow $c_i$ is not involved in the Reidemeister move, then Ind$(c_i)$ and $N(c_i)$ are both preserved. On the other hand, for each case, the three arrows appear in the figure above also have invariant indices and assigned number. As a result, for any arrow $c_j$ in the Gauss diagram, Ind$(c_j)$ and $N(c_j)$ are both invariant under $\Omega_{3a}$.
\end{enumerate}

In conclusion for each $k\geq0$, the polynomial $f_k(t)$ is unchanged under all Reidemeister moves. Therefore it is an invariant of virtual knots.
\end{proof}

\textbf{Remark} We have mentioned that if $K$ is a classical knot diagram then each crossing point of $K$ is even. In fact it is not difficult to conclude that the index of each crossing point is 0 if the knot diagram has no virtual crossings. Hence if $K$ is a classical knot diagram then we have $f_k(t)=0$, which implies that the writhe polynomial $W_K(t)=0.$ Therefore if a virtual knot has non-trivial writhe polynomial we can conclude that it is non-classical and hence non-trivial.

Here we list some simple properties of the writhe polynomial, which can be compared with Proposition 2.3 and the related results in \cite{Kau2012} and \cite{Fol2012}. The proof is almost the same as that of Proposition 2.3, hence it is omitted here.
\begin{proposition}
The writhe polynomial satisfies the following properties
\begin{enumerate}
\item $W_{\overline{K}}(t)=W_K(t^{-1})\cdot t^2$, here $\overline{K}$ denotes the inverse of $K$.
\item $W_{K^*}(t)=-W_K(t^{-1})\cdot t^2$, here $K^*$ denotes the mirror image of $K$.
\item $W_{K_a\#K_b}(t)=W_{K_a}(t)+W_{K_b}(t)$, here $K_a\#K_b$ denotes any connected sum of $K_a$ and $K_b$.
\end{enumerate}
\end{proposition}

Very recently, Louis H. Kauffman introduced a polynomial invariant, the affine index polynomial in \cite{Kau2012}. Another definition of this polynomial using virtual linking number was given in \cite{Fol2012}. Next we will discuss the relation between the writhe polynomial and the affine index polynomial. In order to illuminate the relation between them, we give an alternate definition of the writhe polynomial.

Recall the definition of the writhe polynomial $W_K(t)=\sum\limits_{k=0}^\infty f_k(t)$, where $f_k(t)=\sum\limits_{c_i\in Odd_k(K)}w(c_i)t^{N(c_i)}$. According to the definition of $Odd_k(K)$ we conclude that if an arrow has nonzero index, then it has contribution to some $f_k(K)$, hence to the writhe polynomial. With this viewpoint we can rewrite the writhe polynomial to be
\begin{center}
$W_K(t)=\sum\limits_{\text{Ind}(c_i)\neq0}w(c_i)t^{N(c_i)}=\sum\limits_{c_i}w(c_i)t^{N(c_i)}-\sum\limits_{\text{Ind}(c_i)=0}w(c_i)t^{N(c_i)}$.
\end{center}
According to Lemma 3.2 we have
\begin{center}
$W_K(t)=\sum\limits_{\text{Ind}(c_i)\neq0}w(c_i)t^{\text{Ind}(c_i)+1}=\sum\limits_{c_i}w(c_i)t^{\text{Ind}(c_i)+1}-\sum\limits_{\text{Ind}(c_i)=0}w(c_i)t^{\text{Ind}(c_i)+1}$.
\end{center}

Recall the affine index polynomial introduced in \cite{Kau2012}. Given a virtual knot diagram, by an \emph{arc} we mean an edge from one classical crossing to the next classical crossing according to the orientation. Hence an arc may contain some virtual crossings. Next we want to label each arc with an integer, such that the colorings around a crossing point is indicated as below:
\begin{center}
\includegraphics{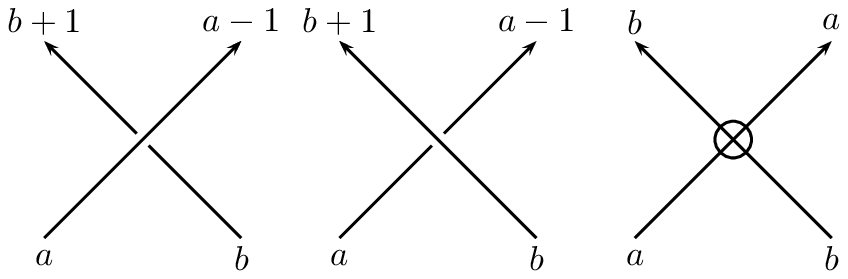}\\
Figure 14
\end{center}
This kind of integer labeling always exists for a virtual knot diagram (including classical knot diagrams) \cite{Kau2012}. Later we will show that this labeling can also be done on classical link diagrams, but not all virtual link diagrams. For example the virtual Hopf link can not be colored in this way. A sufficient and necessary condition of when a virtual link can be colored as above will be given in Section 5.
\begin{center}
\includegraphics{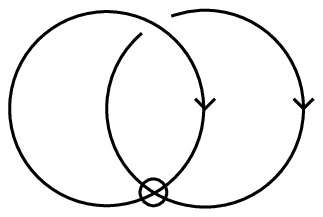}\\
Figure 15
\end{center}
After labeling each arc of a virtual knot diagram, the \emph{weight} of a classical crossing $c$ can be defined as
\begin{center}
$W(c)=\begin{cases}
a-b-1,& \text{if $w(c)=+1$;}\\
b+1-a,& \text{if $w(c)=-1$.}
\end{cases}$
\end{center}
Then for a virtual knot diagram $K$, the \emph{affine index polynomial} of $K$ can be defined as follows
\begin{center}
$P_K(t)=\sum\limits_{c_i}w(c_i)(t^{W(c_i)}-1)=\sum\limits_{c_i}w(c_i)t^{W(c_i)}-wr(K)$,
\end{center}
here $wr(K)=\sum\limits_{c_i}w(c_i)$ denotes the writhe of the diagram. Note that the integer labeling is not unique in general, one labeling will be differ from another by a constant integer. However this affine index polynomial is well defined and is an invariant of virtual knots \cite{Hen2010, Kau2012}.

Now we want to investigate the relation between the writhe polynomial and the affine index polynomial. First let us introduce a simple integer for a given virtual knot diagram $K$, we define a number $Q_K$ by the equation below
\begin{center}
$Q_K=\sum\limits_{\text{Ind}(c_i)\neq0}w(c_i)$.
\end{center}
\begin{lemma}
$Q_K$ is a virtual knot invariant.
\end{lemma}
\begin{proof}
It suffices to show that $Q_K$ is invariant under Reidemeister moves. According to Figure 11, Figure 12, Figure 13 it is evident that $Q_K$ is preserved in each case. The proof is finished.
\end{proof}

With the help of $Q_K$, the relation between the writhe polynomial and the affine index polynomial can be described simply as follows.
\begin{theorem}
Given a virtual knot diagram $K$, we have $W_K(t)=(P_K(t)+Q_K)t$.
\end{theorem}
\begin{proof}
The relation between the writhe polynomial and the affine index polynomial mainly follows from the fact that Ind$(c_i)=W(c_i)$. In fact if we translate the labeling rule from a knot diagram to the corresponding Gauss diagram, it becomes an integer labeling on each arc of the circle. Instead of the coloring introduced during the definition of the odd writhe polynomial, we can introduce another integer assignment to each arc. As before first we choose a point in one arc, then walk along the circle from this point according to the orientation. This time we focus on the crossings which we meet the preimage of the overcrossing earlier than the preimage of the undercrossing. Take the sum of the signs of these crossings, and assign this integer to the arc we choose. It is not difficult to find that this labeling satisfies the labeling rule proposed by L. Kauffman. Notice that for each arc, the sum of the assigned integers from these two labeling rules is exactly the writhe of the diagram. See the figure below, here $w$ denotes the writhe of the diagram.
\begin{center}
\includegraphics{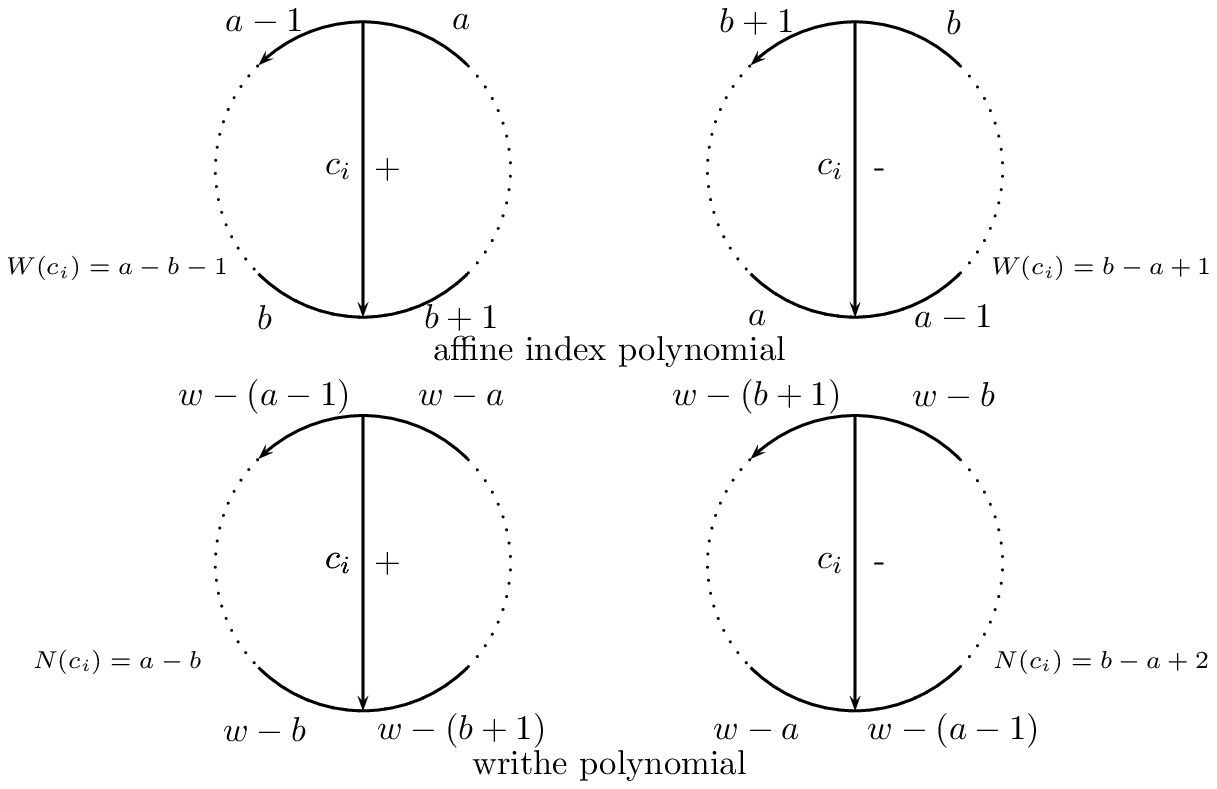}\\
Figure 16
\end{center}
From the figure above we can conclude that $N(c_i)=W(c_i)+1$, combining with Lemma 3.2 we have Ind$(c_i)=W(c_i)$. It follows that
\begin{equation*}
    \begin{array}{rcl}
        W_K(t) & = & \sum\limits_{c_i}w(c_i)t^{\text{Ind}(c_i)+1}-\sum\limits_{\text{Ind}(c_i)=0}w(c_i)t^{\text{Ind}(c_i)+1}\\
         & = & (\sum\limits_{c_i}w(c_i)t^{\text{Ind}(c_i)}-\sum\limits_{\text{Ind}(c_i)=0}w(c_i)t^{\text{Ind}(c_i)})t\\
         & = & (\sum\limits_{c_i}w(c_i)t^{W(c_i)}-\sum\limits_{\text{Ind}(c_i)=0}w(c_i)t^0)t\\
         & = & (\sum\limits_{c_i}w(c_i)t^{W(c_i)}-\sum\limits_{c_i}w(c_i)+\sum\limits_{\text{Ind}(c_i)\neq0}w(c_i))t\\
         & = & (P_K(t)+Q_K)t
    \end{array}
\end{equation*}
\end{proof}

\textbf{Remark} Although the writhe polynomial is the sum of finitely number of nontrivial polynomial invariants, i.e. $W_K(t)=\sum\limits_{k=0}^\infty f_k(t)$, however from the definition of $f_k(t)$ it is not difficult to observe that the writhe polynomial $W_K(t)$ contains the same information as $\{f_0(t), f_1(t), f_2(t), \cdots\}$. In other words with a given writhe polynomial $W_K(t)$, the sequence of polynomial invariants $\{f_0(t), f_1(t), f_2(t), \cdots\}$ can be obtained from $W_K(t)$.

According to Theorem 3.6, we know that the writhe polynomial is essentially equivalent to the affine index polynomial (and the wriggle polynomial \cite{Fol2012}). In \cite{Kau2012} the Vassiliev invariants obtained from the affine index polynomial are given. Besides, L. Kauffman also discussed the affine index polynomial with the viewpoint of flat biquandle. The behavior of wriggle polynomial (which is equivalent to the affine index polynomial) under mutation is considered in \cite{Fol2012}. See \cite{Kau2012, Fol2012} for more interesting examples and applications of the affine index polynomial.

\section{The writhe polynomial on long flat virtual knots}
A long knot is a 1-1 tangle with one input end and one output end. Given a long knot we can get a knot (its closure) by connecting the two ends of the long knot. Conversely a long knot can be obtained by removing one point from a knot. In classical knot theory a long knot is knotted if and only if its closure is knotted. And the place where the point is removed is not important. However in virtual knot theory this is not the case. For example a nontrivial long virtual knot may have trivial closure. For more information about long virtual knot the reader is referred to \cite{Gou2000,YHI2010,YHI2012,Ish2010,Kau2011,Man2004,Man2004'}.

Given a long virtual knot diagram $\widetilde{K}$, by attaching the ends we obtain a virtual knot $C(\widetilde{K})$, i.e. the closure of $\widetilde{K}$. Hence we can define the writhe polynomial $W_{C(\widetilde{K})}(t)$, obviously $W_{C(\widetilde{K})}(t)$ is an invariant of $\widetilde{K}$. According to Proposition 3.4, $W_{C(\widetilde{K})}(t)$ can be used to distinguish $\widetilde{K}$ from its inverse or mirror image. However it is a rather ``weak" invariant of long virtual knots in some sense. For example in virtual knot theory, there exists a pair of different long virtual knots which have the same closure. In this case the polynomial $W_{C(\widetilde{K})}(t)$ can not tell the difference between these long virtual knots.

A \emph{long flat virtual knot} is a long virtual knot without the classical crossing information. Therefore there are only two kinds of crossing points: the virtual crossing, and the flat crossing. A flat crossing is usually indicated by transversely intersecting line segments. And the Reidemeister moves in long flat virtual knot are exactly the moves given in Figure 2 where all classical crossings are replaced by flat crossings. Without loss of generality we always assume that the the long knot is oriented from left to right. It is evident that if there is no virtual crossing points on a long flat virtual knot diagram, then the long flat virtual knot is trivial, i.e. equivalent to a straight line.

Given a long flat virtual knot, by replacing each flat crossing by an overcrossing or an undercrossing, one obtains a long virtual knot. If the long flat virtual knot has $n$ flat crossings then it corresponds to $2^n$ long virtual knot diagrams totally. Conversely, with a given long virtual knot $\widetilde{K}$, by forgetting the crossing information of each classical crossing point we obtain a long flat virtual knot $F(\widetilde{K})$. According to \cite{Kau2011} we say the long virtual knot \emph{overlies} the associated long flat virtual knot. Obviously if $F(\widetilde{K})$ is nontrivial then $\widetilde{K}$ is necessary nontrivial. Hence the non-triviality of a long flat virtual knot induces the non-trivialities of all the associated long virtual knots. Note that a long flat virtual knot is trivial if and only if it is classical (contains no virtual crossings), hence the non-triviality of a long flat virtual knot implies that all long virtual knots that overlie it are nonclassical. Therefore finding some nontrivial invariants of the long flat virtual knots is always very helpful.

Let $K$ be a long flat virtual knot, and $\widetilde{K}$ a long virtual knot that overlies $K$, i.e. $K=F(\widetilde{K})$. Then the ``Gauss diagram" of $\widetilde{K}$, say $G(\widetilde{K})$, can be constructed similarly. Instead of the circle, now the preimage of the long virtual knot is still a real line, which is also oriented from the left side to the right side. As before for each classical crossing point there is an arc $($we still call it a arrow$)$ oriented from the preimage of the overcrossing to the preimage of the undercrossing. Here each arrow is always been drawn as an oriented semi-circle above the real line. The figure below gives an example.
\begin{center}
\includegraphics{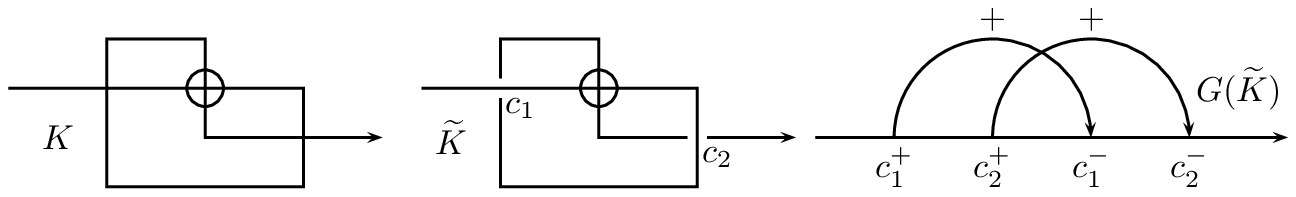}\\
Figure 17
\end{center}

For an arrow $c_i$ in $G(\widetilde{K})$, we will label two integers to it, $o(c_i)$ and $I(c_i)$. The integer $o(c_i)$ can be defined as follows, if the orientation of $c_i$ is the same as the orientation of the long line, then we define $o(c_i)=1$, otherwise $o(c_i)=-1$.
\begin{center}
\includegraphics{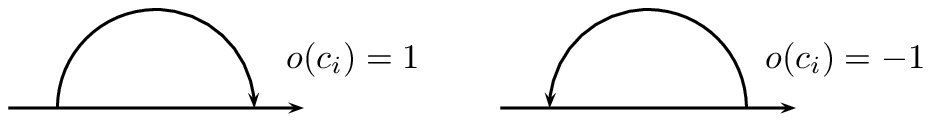}\\
Figure 18
\end{center}

To define $I(c_i)$, let us consider the arrows that cross $c_i$ transversely in the Gauss diagram. According to the orientations of these arrows, some arrows are oriented from the inside of $c_i$ to the outside of $c_i$, and the others are oriented from the outside of $c_i$ to the inside of $c_i$. Here the inside of $c_i$ we mean the segment between the $c_i^+$ and $c_i^-$, and $c_i^+$ $(c_i^-)$ denotes the preimage of the overcrossing $($undercrossing$)$ of $c_i$. Then we use $u_+(c_i)$ $(u_-(c_i))$ to denote the number of positive $($negative$)$ arrows that cross $c_i$ from inside to outside, and $d_+(c_i)$ $(d_-(c_i))$ to denote the number of positive $($negative$)$ arrows that cross $c_i$ from outside to inside. Now we can define the \emph{index} Ind$(c_i)$ $($or $I(c_i)$ for short$)$ of $c_i$ by the equality below
\begin{center}
$I(c_i)=u_+(c_i)-u_-(c_i)-d_+(c_i)+d_-(c_i).$
\end{center}
Note that $I(c_i)$ does not depend on the value of $o(c_i)$.
\begin{center}
\includegraphics{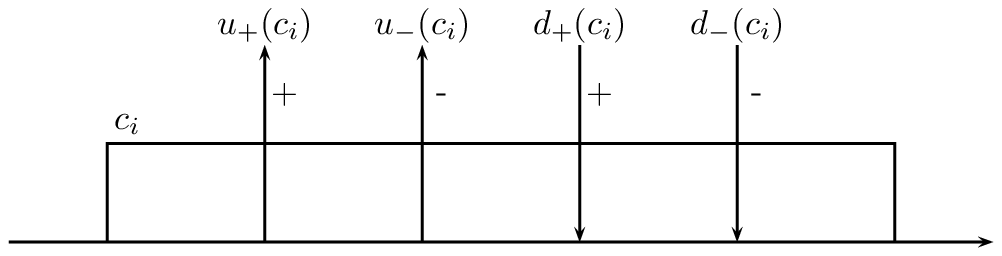}\\
Figure 19
\end{center}

Now we can define the \emph{writhe polynomial} for the long flat virtual knot $K$, which is a bit different from the writhe polynomial of a virtual knot defined in Section 3. The writhe polynomial of the long flat virtual knot $K$ is defined to be
\begin{center}
$W_K(t)=\sum\limits_{I(c_i)\neq0}o(c_i)w(c_i)t^{I(c_i)}$.
\end{center}

\begin{theorem}
The writhe polynomial defined above is a well defined polynomial invariant of long flat virtual knots.
\end{theorem}

\begin{proof}
First we need to show that $W_K(t)$ is well defined, i.e. it does not depend on the choice of $\widetilde{K}$. In order to show this, it suffices to prove that $W_K(t)$ is invariant under switching any crossing point of $\widetilde{K}$. Assume $\widetilde{K'}$ is another long virtual knot which can be obtained from $\widetilde{K}$ by switching one classical crossing $c$, obviously $\widetilde{K'}$ also overlies $K$. Notice that switching the crossing point $c$ will change the writhe $w(c)$ and $o(c)$, but preserve $I(c)$. For any other arrow $c'$ in the Gauss diagram, $w(c')$ and $o(c')$ are unchanged. If $c\cap c'\neq\emptyset$, i.e. , since $w(c)$ and $o(c)$ are both changed then the contribution of $c$ to $I(c')$ in $G(\widetilde{K})$ is equivalent to that in $G(\widetilde{K'})$. As a result the writhe polynomial defined from $\widetilde{K'}$ is the same as the polynomial defined from $\widetilde{K}$. It follows that the writhe polynomial is a well defined polynomial.

Next we want to prove that the writhe polynomial is an invariant of long flat virtual knots. It is sufficient to show that the writhe polynomial is invariant under all Reidemeister moves in long flat virtual knot theory. Since the writhe polynomial is well defined, hence for a pair of long flat virtual knots which are related by one Reidemeister move we just need to consider a pair of long virtual knots which overlie that two long flat virtual knots, and they can be obtained from each other by the corresponding Reidemeister move in long virtual knot theory. Let us consider the Reidemeister moves given in figure 10.
\begin{enumerate}
\item $\Omega_{1a}$ on a knot diagram corresponds to adding or removing an isolated arrow hence with trivial index on the Gauss diagram. Obviously for any other arrow $c$, the triple $\{w(c), o(c), I(c)\}$ is preserved, it follows that the writhe polynomial is invariant under $\Omega_{1a}$. For $\Omega_{1b}$ the analysis is the same.
\item For the second Reidemeister move $\Omega_{2a}$, clearly on the Gauss diagram the corresponding move is adding or removing a pair of arrows, say $c_1$ and $c_2$. From figure 12 we conclude that
\begin{center}
$o(c_1)=o(c_2)$, $w(c_1)=-w(c_2)$ and $I(c_1)=I(c_2)$.
\end{center}
Hence their contributions to the writhe polynomial will be canceled. On the other hand since they have the same orientation but different writhes, the triple $\{w(c), o(c), I(c)\}$ for any other arrow $c$ is invariant. Therefore the writhe polynomial is also preserved under $\Omega_{2a}$.
\item Similarly, the third Reidemeister move $\Omega_{3a}$ also preserves the writhe polynomial. In order to see this we just need to notice that for each arrow $c$ on the Gauss diagram, the triple $\{w(c), o(c), I(c)\}$ is kept under $\Omega_{3a}$. Hence the writhe polynomial is also invariant under $\Omega_{3a}$.
\end{enumerate}

In conclusion, we have shown that the writhe polynomial $W_K(t)=\sum\limits_{I(c_i)\neq0}o(c_i)w(c_i)t^{I(c_i)}$ does not depend on the choice of the long virtual knot that overlies $K$, and it is preserved under all Reidemeister moves. The proof is finished.
\end{proof}

\textbf{Remark} Replacing $t$ by 1 in the writhe polynomial we get $W_K(1)=\sum\limits_{I(c_i)\neq0}o(c_i)w(c_i)$. Obviously this is an integer invariant of long flat virtual knots. We remark that the relation between $\sum\limits_{I(c_i)\neq0}o(c_i)w(c_i)$ and the writhe polynomial in long flat virtual knot theory is analogous to the relation between $Q_K$ and the writhe polynomial in virtual knot theory.

Now we have defined a polynomial invariant $W_K(t)$ for a long flat virtual knot $K$. A natural question is: are there any relations between the writhe polynomial of a long flat virtual knot $K$ and the writhe polynomial of long virtual knots that overlie $K$?

As we have mentioned before, given a long flat virtual knot $K$, there are a total of $2^n$ long virtual knots that overlie $K$. Here $n$ denotes the number of classical crossing points of $K$. Let us consider the descending long virtual knot diagram $D(K)$. By a \emph{descending} long virtual knot diagram, we mean the long virtual knot diagram which is obtained from $K$ by creating a classical crossing point at each flat crossing such that if we traverse $K$ according to the orientation $($from its left end to the right end$)$, for each classical crossing we will pass over the crossing before passing under the crossing. Meanwhile all virtual crossing points are kept. Obviously for each arrow $c_i$ in the Gauss diagram of $D(K)$ we have $o(c_i)=1$. Consider the closure of the long virtual knot $D(K)$, we denote it by $C(D(K))$. Then the index of each arc in the Gauss diagram of $D(K)$ exactly coincides with the index of the corresponding arrow in the Gauss diagram of $C(D(K))$, and the writhe of each crossing of $D(K)$ equals to the writhe of the corresponding crossing in $C(D(K))$. It follows that
\begin{center}
$W_K(t)=\sum\limits_{I(c_i)\neq0}o(c_i)w(c_i)t^{I(c_i)}=\sum\limits_{\text{Ind}(c_i)\neq0}w(c_i)t^{\text{Ind}(c_i)}=W_{C(D(K))}(t)\cdot t^{-1}$.
\end{center}
Hence the writhe polynomial of a long flat virtual knot $K$ is essentially equivalent to the writhe polynomial of the closure of the descending long virtual knot of $K$. That is the reason why we still call it the writhe polynomial. From this point of view we can easily conclude that $W_K(t)=-W_{\overline{K}}(t)$ from Proposition 3.4, here $\overline{K}$ denotes the inverse of the long flat virtual knot $K$. Hence the writhe polynomial sometimes can distinguish a long flat virtual knot from its inverse. We remark that the definition with this viewpoint can be regarded as an application of the Long Flat Embedding Theorem in \cite{Kau2011}.

For a long virtual knot with trivial closure. It is unable to distinguish whether this long virtual knot is non-classical or not by studying its closure. However sometimes we can prove it is non-classical by showing the associated long flat virtual knot is nontrivial. For example the long virtual knot $\widetilde{K}$ in figure 20 has a trivial closure $C(\widetilde{K})$, hence all invariants obtained from $C(\widetilde{K})$ are trivial. However the associated long flat virtual knot $F(\widetilde{K})$ is nontrivial, since $W_{F(\widetilde{K})}(t)=t+t^{-1}$. It follows that $\widetilde{K}$ is nonclassical and hence nontrivial. In fact since the odd writhe of $D(F(\widetilde{K}))$ is 2, therefore $F(\widetilde{K})$ is nontrivial which also implies that $\widetilde{K}$ is nonclassical.
\begin{center}
\includegraphics{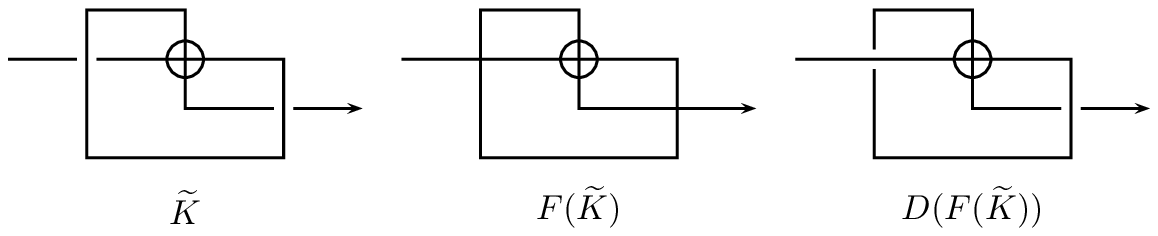}\\
Figure 20
\end{center}

Next we give another example of long virtual knots, for convenience here we only give the Gauss diagrams of them.
\begin{center}
\includegraphics{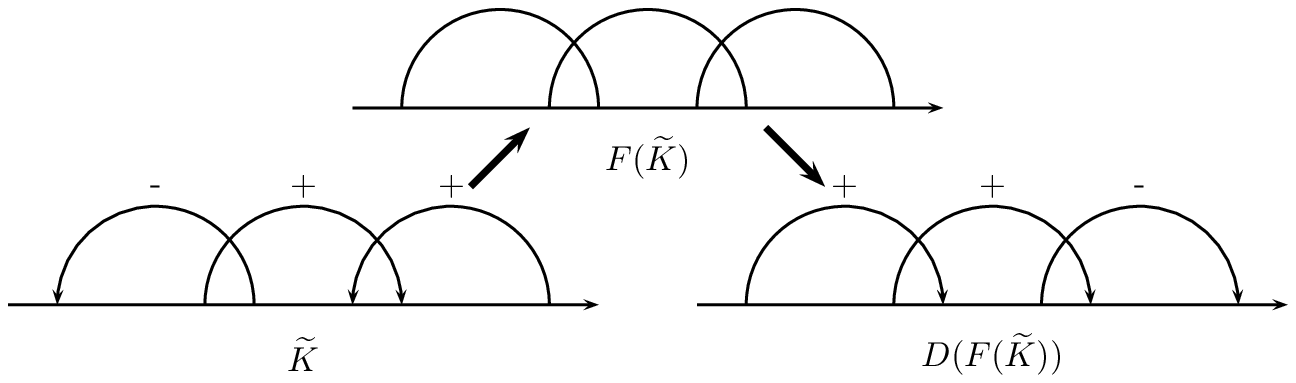}\\
Figure 21
\end{center}
First we have a long virtual knot $\widetilde{K}$, and the closure of $\widetilde{K}$ has trivial odd writhe. Consider the associated long flat virtual knot $F(\widetilde{K})$ and its descending long virtual knot diagram $D(F(\widetilde{K}))$. Now the odd writhe of $D(F(\widetilde{K}))$ is trivial hence the odd writhe does not work here. However, we have $W_{F(\widetilde{K})}=t-t^{-1}+t^{-2}$. As a consequence, the long virtual knot $\widetilde{K}$ is nonclassical and nontrivial.

Note that although the connected sum is not well defined for virtual knots, it is well defined for long virtual knots and long flat virtual knots. For a pair of long virtual knots $\widetilde{K_1}$ and $\widetilde{K_2}$, the \emph{connected sum} $\widetilde{K_1}\#\widetilde{K_2}$ can be obtained by gluing the right end of $\widetilde{K_1}$ to the left end of $\widetilde{K_2}$. For the long flat virtual knots the construction is similar. The following proposition is obvious.
\begin{proposition}\textcolor[rgb]{1.00,1.00,1.00}{c}
\begin{enumerate}
  \item Let $\widetilde{K_1}$ and $\widetilde{K_2}$ be a pair of long virtual knots, then $W_{C(\widetilde{K_1}\#\widetilde{K_2})}(t)=W_{C(\widetilde{K_1})}(t)+W_{C(\widetilde{K_2})}(t)$.
  \item Let $K_1$ and $K_2$ be a pair of long flat virtual knots, then $W_{K_1\#K_2}(t)=W_{K_1}(t)+W_{K_2}(t)$.
\end{enumerate}
\end{proposition}

Crossing number is a basic invariant in classical knot theory. It is well known that in virtual knot theory there are three kinds of crossing number: the real crossing number, the virtual crossing number and the crossing number. The \emph{real crossing number} $c_r(K)$ is the minimal number of real crossing points of all diagrams representing $K$. Similarly the \emph{virtual crossing number} $c_v(K)$ is the minimal number of virtual crossings of all diagrams representing $K$. The \emph{crossing number} $c(K)$ can be defined to be the minimal number of crossings $($real and virtual$)$ of all diagrams representing $K$. As in the classical knot theory, it is a challenging task to determine each crossing number of a virtual knot. Some lower bounds of the real crossing number and virtual crossing number are obtained by studying the arrow polynomial and the Miyazawa polynomial \cite{Dye2009,Miy2008}. In \cite{Che2012} some information of the real crossing number are also given with the odd writhe polynomial. Analogously, the writhe polynomial of a virtual knot also can offer some information about the real crossing number.

Given a long flat virtual knot $K$, there are only two kinds of crossing point, the flat crossing and the virtual crossing. We define the \emph{flat crossing number} of $K$ to be the minimal number of flat crossing points of all diagrams representing $K$. Let $c_f(K)$ denote the flat crossing number of $K$, according to the relation between a long virtual knot and the associated long flat virtual knot we have
\begin{center}
$c_f(K)\leq c_r(\widetilde{K})$,
\end{center}
here $\widetilde{K}$ denotes an arbitrary long virtual knot which overlies $K$. Hence the flat crossing number of $K$ gives a lower bound of the real crossing number for all long virtual knots that overlie $K$. See \cite{YHI2010,Ish2010} for more information about the crossing number of long virtual knots.

Consider the writhe polynomial of a long flat virtual knot $K$, let $s(K)$ denote the sum of the absolute values of the coefficients of $W_K(t)$. Equivalently speaking, if $W_K(t)=a_nt^n+a_{n+1}t^{n+1}+\cdots+a_{m}t^{m}$ $(n\leq m)$, here $n$ and $m$ denote the lowest degree and the highest degree of the polynomial respectively. Then we define $s(K)$ by the equation below
\begin{center}
$s(K)=|a_n|+|a_{n+1}|+\cdots+|a_m|.$
\end{center}

\begin{proposition}
Given a long flat virtual knot $K$, we have $s(K)\leq c_f(K)$.
\end{proposition}
\begin{proof}
According to the definition of the writhe polynomial, the contribution of each flat crossing of $K$ to $s(K)$ is at most one. Hence the result follows.
\end{proof}

For example let us revisit the long flat virtual knot $K$ in figure 17. Its writhe polynomial is $t+t^{-1}$, hence the flat crossing number of $K$ is exactly 2. As a consequence the real crossing numbers of the four associated long virtual knots are also 2. Similarly the writhe polynomial of the long flat virtual knot $F(\widetilde{K})$ in figure 21 is $t-t^{-1}+t^{-2}$, it follows that $c_f(F(\widetilde{K}))=c_r(\widetilde{K})=3$.

\section{The linking polynomial of virtual links}
In Section 3 we introduce the writhe polynomial invariant, we name it in this way because it can be regarded as a generalization of the writhe, although the writhe is not an invariant of virtual knots. Or strictly speaking we should regard the writhe polynomial invariant as a generalization of the integer invariant $Q_K$. The aim of this section is to generalize the linking number of a 2-component virtual link to a polynomial invariant, we name it the linking polynomial. The relation between the linking polynomial and the linking number will be discussed in detail later.

In this section when we talk about a virtual link, we always mean a 2-component virtual link. For links with more components the linking polynomial can be defined for each pair of the components, which is similar to the classical case. Given a 2-component $($virtual$)$ link $L=K_1\cup K_2$, we will always abuse the notation, letting $L=K_1\cup K_2$ refer both to a $($virtual$)$ link diagram and to the $($virtual$)$ link itself. If $L$ is classical, then the linking number of $L$ has many different definitions according to different points of view \cite{Rol1976}. Here we define the \emph{linking number} of $L=K_1\cup K_2$ as below
\begin{center}
$lk(L)=\frac{1}{2}\sum\limits_{c_i\in K_1\cap K_2}w(c_i)$,
\end{center}
here $w(c_i)$ denotes the writhe of $c_i$. For a classical link $L$, the linking number of $L$ is always an integer. According to the definition of linking number, $lk(L)$ is invariant if one switches a self-crossing of $L$, here we say a crossing $c$ of $L$ is a \emph{self-crossing point} if $c\in K_1\cap K_1$ or $c\in K_2\cap K_2$. Note that throughout this paper we always abuse the notation $K_i\cap K_i$ $(i=1, 2)$ to denote the set of self-crossings of $K_i$. The reader should not confuse it with the usual set theory notation. We remark that with the viewpoint of parity, we refer to a crossing point as \emph{even} if it is a self-crossing point, otherwise it is \emph{odd}. Then for a given classical or virtual 2-component link, this parity satisfies the parity axioms given in Section 3 \cite{Man2010}.

Given a 2-component virtual link, the linking number can be defined similarly. Let $L=K_1\cup K_2$ be a 2-component virtual link diagram, then we define the \emph{linking number} of $L$ by the equation below
\begin{center}
$lk(L)=\frac{1}{2}\sum\limits_{\text{classical }c_i\in K_1\cap K_2}w(c_i)$.
\end{center}
When $L$ is a classical link, this definition coincides with the definition above. It is easy to observe that the linking number of a virtual link is also an invariant. However if $L$ is not classical, $lk(L)$ need not to be an integer. For example, consider the virtual Hopf link below. Its linking number is $\frac{1}{2}$, hence the virtual Hopf link is nonclassical and nontrivial.
\begin{center}
\includegraphics{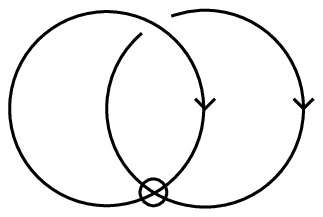}\\
Figure 22
\end{center}

Before giving our construction of the linking polynomial, let us take a short review about the related result in \cite{Kau2012}. In \cite{Kau2012} L. Kauffman defined the affine index polynomial for virtual knots. Recall the coloring in figure 14, for a virtual knot this coloring is always valid. For instance, the coloring which is used in Theorem 3.6 satisfies this labeling rules. However the coloring given in figure 14 does not always exists for virtual links. So if one wants to generalize the affine index polynomial to 2-component links, the first thing is to discuss for which kind of 2-component link this coloring is still valid.

Given a 2-component virtual link $L=K_1\cup K_2$, its Gauss diagram can be constructed similarly. Since $L$ is a 2-component link, the Gauss diagram of $L$ consists of two circles. Each crossing $c\in K_1\cap K_1$ or $K_2\cap K_2$ corresponds to an arrow with two ends on one circle, and each crossing $c\in K_1\cap K_2$ corresponds to an arc connecting the two circles. As before every arrow or arc is oriented from the preimage of the overcrossing to the preimage of the undercrossing. The figure below gives the Gauss diagram of the virtual Hopf link.
\begin{center}
\includegraphics{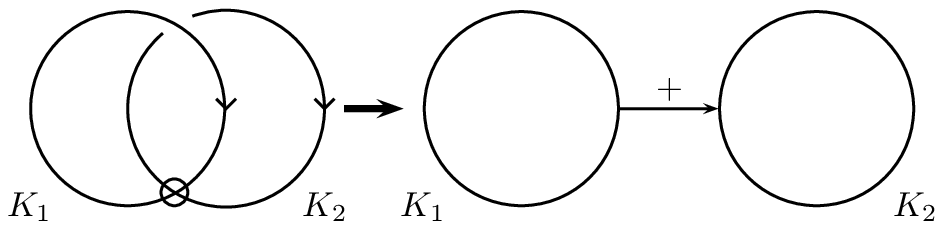}\\
Figure 23
\end{center}

Given a virtual link, consider the Gauss diagram of it. Let $r_+$ $(r_-)$ denote the number of positive $($negative$)$ arcs that are oriented from $K_1$ to $K_2$, and $l_+$ $(l_-)$ denote the number of positive $($negative$)$ arcs with orientation from $K_2$ to $K_1$. Then the linking number of $L$ can be rewritten as
\begin{center}
$2lk(L)=r_+-r_-+l_+-l_-$.
\end{center}
Now we define another integer as
\begin{center}
$span(L)=|r_+-r_--l_++l_-|$.
\end{center}
We call it the \emph{span} of the virtual link. Note that $span(L)$ is exactly the absolute value of the wriggle number introduced in \cite{Fol2012}. The following lemma is evident, which can be also found in \cite{Fol2012}.
\begin{lemma}
The span is a virtual link invariant. In particular $span(L)=0$ if $L$ is classical.
\end{lemma}

Now we can give an answer to the question above, in fact we have the following proposition
\begin{proposition}
A 2-component virtual link $L$ can be labeled with integers in the style of figure 14 if and only if $span(L)=0$.
\end{proposition}
\begin{proof}
If $L$ has a coloring as assumed, consider one component of $L$, for example the first component $K_1$. Choose a point of $K_1$ and traverse $K_1$ according to the orientation. By figure 14,  every time when we meet an overcrossing with sign $\pm 1$ the labeled integer will be increased by $\mp 1$, and if we meet an undercrossing with sign $\pm 1$ the labeled integer will be increased by $\pm 1$. Since this labeling exists, when we go back to the original point the labeling integer is preserved. Notice that the contribution coming from all self-crossings of $K_1$ will be killed mutually. It follows that $r_+-r_--l_++l_-=0$.

Now we assume that $span(L)=0$, it suffices to construct a labeling in the manner of figure 14. First let us choose a point on the diagram of $K_1$, and label the associated arc with an integer $p$. Other arcs will be labeled according to the labeling rule given in figure 14. Since $span(L)=|r_+-r_--l_++l_-|=0$, this labeling is well defined. For the other component, the arcs can be labeled analogously.
\end{proof}

From the proposition above we know that the coloring introduced in \cite{Kau2012} exists if and only if the span of the link is trivial. In particular for a classical link diagram this coloring always exists. However it is easy to observe that the colorings of the two components are independent. For example the figure below gives two different coloring of the Hopf link, the associated affine index polynomial are also presented.
\begin{center}
\includegraphics{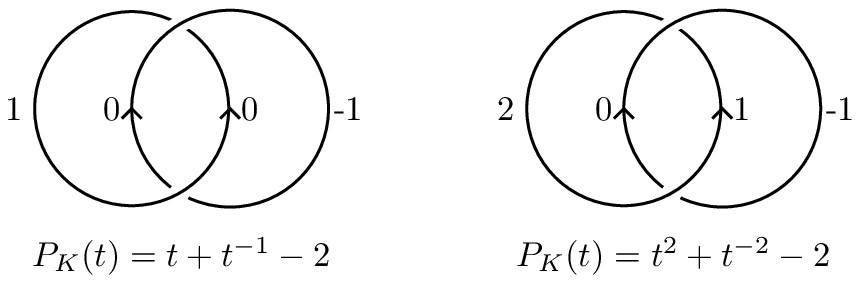}\\
Figure 24
\end{center}
Since the coloring of a link diagram $L$ with $span(L)=0$ is not unique, in general the affine index polynomial is not well defined for links even if the coloring in Figure 14 is valid. The rest of this section is devoted to give an alternative way to define a polynomial invariant of 2-component virtual links.

Given a 2-component virtual link $L=K_1\cup K_2$, there are two circles in the Gauss diagram of $L$, i.e. the preimage of $K_1$ and $K_2$ respectively. Consider the universal cover of these two circles, which are just two real lines, we use $l_1$ and $l_2$ to denote them respectively. Without loss of generality we assume that the orientations of $l_1$ and $l_2$ are both indicated from left to right. Notice that each arrow or arc in the Gauss diagram has no preferred lift to the universal cover, however the two ends of each arrow or arc can be locally lifted to the universal cover. The figure below gives an example of the virtual Hopf link.
\begin{center}
\includegraphics{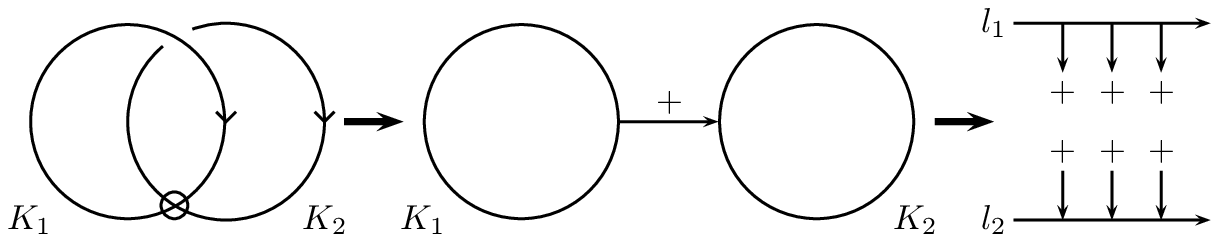}\\
Figure 25
\end{center}
Now choose any segment of $l_1$ and label an integer $p$ to it. For other segments of $l_1$ there is an algorithm to color them with integers. Let us walk along $l_1$ according to the orientation,
\begin{itemize}
  \item when we meet the head of an arrow with positive sign, the next integer will be decreased by one;
  \item when we meet the tail of an arrow with negative sign, the next integer will be decreased by one;
  \item when we meet the head of an arrow with negative sign, the next integer will be increased by one;
  \item when we meet the tail of an arrow with positive sign, the next integer will be increased by one.
\end{itemize}
\begin{center}
\includegraphics{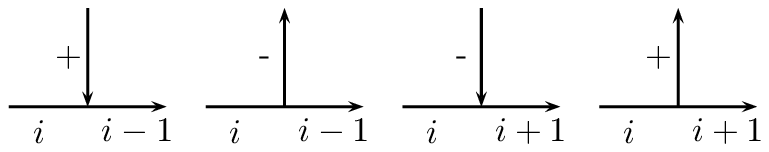}\\
Figure 26
\end{center}
Obviously this labeling is not unique, two labeling are different by a constant integer on each segment. Similarly the second line $l_2$ can also be colored in this manner, and the colorings of $l_1$ and $l_2$ are irrelevant. Then we replace each labeled number $n$ by its equivalent class $[n]\in \mathbb{Z}_{span(L)}$. In particular if $span(L)=0$ then $[n]=n$. If $span(L)=1$ then all labeled numbers are 0, later we will find that nothing interesting can be obtained in this case.

Next we assign an integer $N(c_i)$ to each arc $c_i$ between $l_1$ and $l_2$, just as what we did when we defined the odd writhe polynomial in Section 2. Because there is no preferred choice of lifting for each arc between the two circles, first we choose a lift for each arc arbitrarily. Then we label an integer to each lifted arc according to the figure below.
\begin{center}
\includegraphics{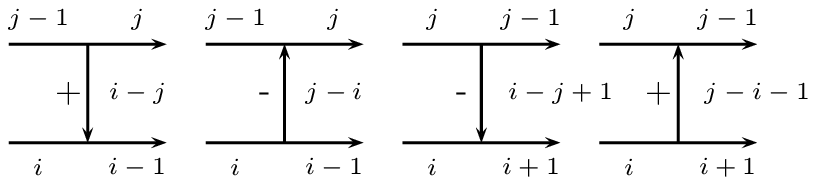}\\
Figure 27
\end{center}
Notice that for a fixed labeling of $l_1$ and $l_2$, the assigned number $N(c_i)$ is well defined. In other words it does not depend on the choice of the lift. This is because the difference of $N(c_i)$ induced by different lifts of $c_i$ is a multiple of $span(L)$. Indeed for a fixed lift of one end of $c_i$, the difference of $N(c_i)$ between two adjacent lifts of the other end of $c_i$ is exactly the span of $L$. Hence different lifts give the same $N(c_i)$ in $\mathbb{Z}_{span(L)}$.

The lifted arcs between $l_1$ and $l_2$ fall into two types: the first type is oriented from $l_1$ to $l_2$, the second type is oriented from $l_2$ to $l_1$. Let $F(t)$ denote the sum of $w(c_i)t^{N(c_i)}$ for all arcs of the first type, and $G(t)$ be the sum of $w(c_i)t^{N(c_i)}$ for all arcs of the second type. Equivalently speaking, we define
\begin{center}
$F(t)=\sum\limits_{c_i:\ l_1\rightarrow l_2}w(c_i)t^{N(c_i)}$, $G(t)=\sum\limits_{c_j:\ l_2\rightarrow l_1}w(c_j)t^{N(c_j)}$.
\end{center}
The following equations are evident.
\begin{center}
$F(1)+G(1)=2lk(L)$, $|F(1)-G(1)|=span(L)$.
\end{center}
Note that since the labeling of each segment of $l_1$ and $l_2$ is not unique, $F(t)$ and $G(t)$ are only determined up to multiplication by the Laurent monomial $t^{\pm k}$.

\begin{theorem}
$F(t)$ and $G(t)$ satisfy the following properties:
\begin{enumerate}
  \item Given a fixed coloring of $l_1$ and $l_2$ respectively, the polynomials $F(t)$ and $G(t)$ are both invariant under all Reidemeister moves.
  \item Given two different colorings of $\{l_1, l_2\}$, then there are two pairs of $\{F(t), G(t)\}$, say $\{F_1(t), G_1(t)\}$ and $\{F_2(t), G_2(t)\}$. If $F_2(t)=F_1(t)\cdot t^k$, then $G_2(t)=G_1(t)\cdot t^{-k}$.
\end{enumerate}
\end{theorem}

\begin{proof}
Let us prove the two properties respectively.

\begin{enumerate}
\item As we mentioned above for a fixed coloring of $l_1$ and $l_2$, $F(t)$ and $G(t)$ are both well defined. It suffices to prove that $F(t)$ is preserved under all Reidemeister moves. For $G(t)$ the proof is analogous.

First let us consider the first Reidemeister move $\Omega_{1a}$, see figure 10. Obviously the crossing involved in the Reidemeister move is a self-crossing. Without loss of generality we assume that it is a self-crossing of $K_1$. Then the figure below shows the effect of $\Omega_{1a}$ to the universal cover of the Gauss diagram. It is easy to find that the polynomial $F(t)$ is invariant under $\Omega_{1a}$. Similarly we can prove that the linking polynomial is also preserved under $\Omega_{1b}$.
\begin{center}
\includegraphics{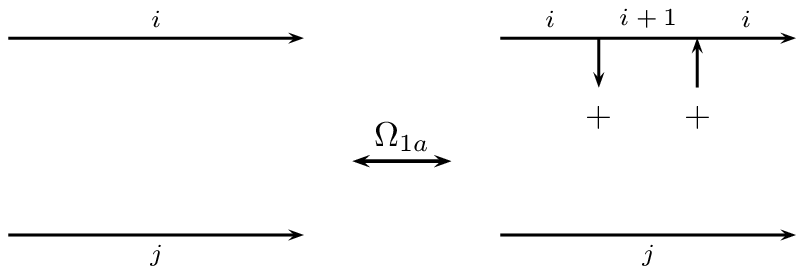}\\
Figure 28
\end{center}

Now we turn to prove that $F(t)$ is invariant under $\Omega_{2a}$. There are two possibilities for the two crossing points involved in the Reidemeister move: both two crossings are self-crossings, or both of them are not self-crossings. Let us consider the first case. Without loss of generality we suppose that these two crossings are two self-crossings of $K_1$, see the figure below. It is easy to observe that $F(t)$ is invariant in this case.
\begin{center}
\includegraphics{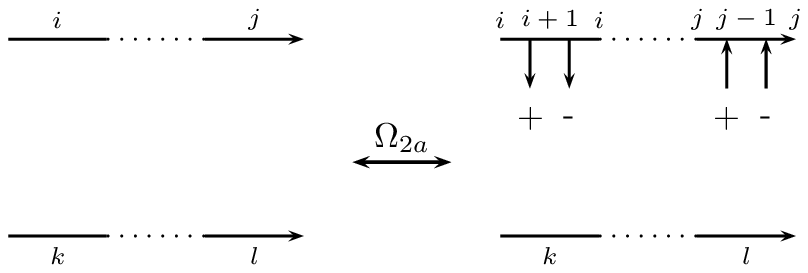}\\
Figure 29
\end{center}
Next we consider the second case, i.e. the two crossing points are intersected by $K_1$ and $K_2$. In this case, the effect of $\Omega_{2a}$ on the universal cover is given as below. Note that the lift of each arc is arbitrary, hence we can choose the lift as below.
\begin{center}
\includegraphics{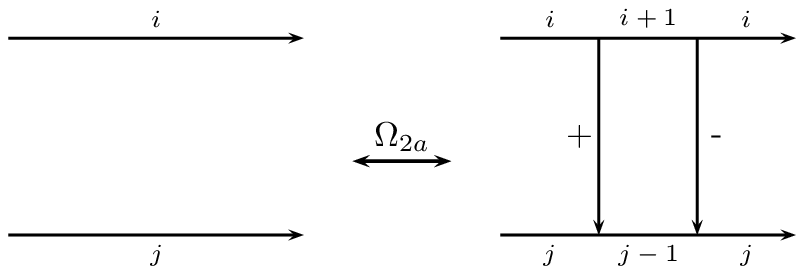}\\
Figure 30
\end{center}
According to the figure above, the Reidemeister move $\Omega_{2a}$ will increase or decrease a pair of arcs oriented from $l_1$ to $l_2$, say $c_1$ and $c_2$. It can be read from the figure that $N(c_1)=N(c_2)=j-i-1$ and $w(c_1)w(c_2)=-1$. Hence their contribution to $F(t)$ is $t^{j-i-1}-t^{j-i-1}=0$. It follows that $F(t)$ is invariant under $\Omega_{2a}$.

Finally we need to prove that $F(t)$ is invariant under the third Reidemeister move $\Omega_{3a}$. If all the three crossings involved in $\Omega_{3a}$ are self-crossing points, then for each crossing point $c\in K_1\cap K_2$ the monomial $w(c)t^{N(c)}$ is kept, hence $F(t)$ and $G(t)$ are both invariant under $\Omega_{3a}$. So it is sufficient to discuss the case that not all the three crossings are self-crossing points. We use $a, b, c$ to denote the three curves involved in the $\Omega_{3a}$. Our discussion divides into three cases.
\begin{center}
\includegraphics{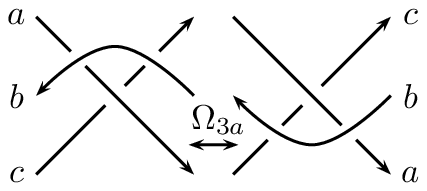}\\
Figure 31
\end{center}
\begin{itemize}
\item Case 1: $a, b\in K_2$ and $c\in K_1$. The figure below gives the associated transformation on the universal cover of the Gauss diagram. It is evident that for each crossing $c_i\in K_1\cap K_2$, the sign $w(c_i)$ and the assigned integer $N(c_i)$ are both invariant. For example the contribution of the two crossing points between $K_1$ and $K_2$ involved in the Reidemeister move is $t^{i-j}-t^{i-k}$ on both sides. Hence $F(t)$ and $G(t)$ are both invariant.
\begin{center}
\includegraphics{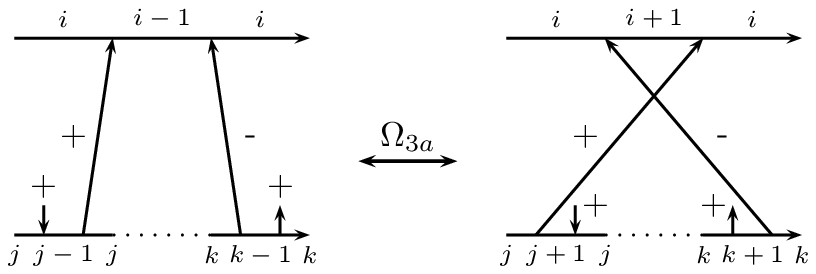}\\
Figure 32
\end{center}
\item Case 2: $a, c\in K_2$ and $b\in K_1$. Similarly for each crossing $c_i\in K_1\cap K_2$, $w(c_i)$ and $N(c_i)$ are also invariant in this case. It follows that $F(t)$ and $G(t)$ are preserved.
\begin{center}
\includegraphics{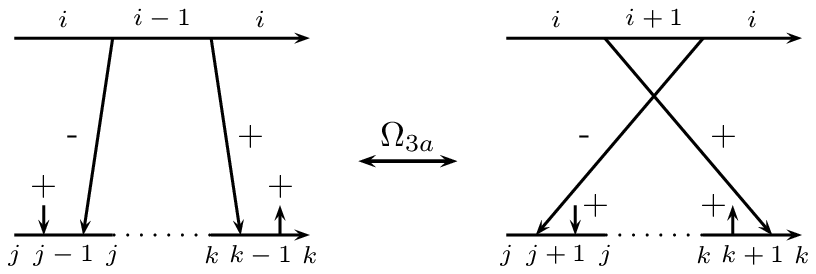}\\
Figure 33
\end{center}
\item Case 3: $b, c\in K_2$ and $a\in K_1$. In this case the discussion is almost the same, both $F(t)$ and $G(t)$ are kept under $\Omega_{3a}$.
\begin{center}
\includegraphics{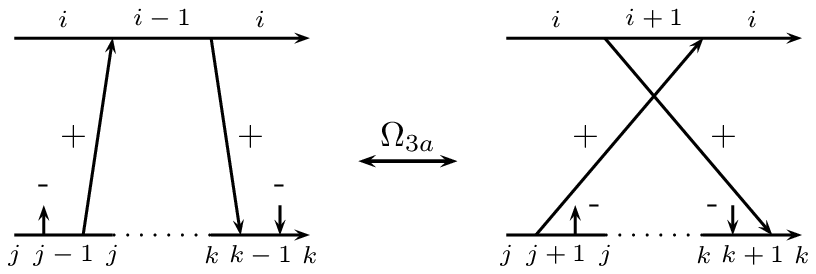}\\
Figure 34
\end{center}
\end{itemize}
In conclusion we have proven that $F(t)$ and $G(t)$ are both invariant under all Reidemeister moves for a fixed coloring of $l_1$ and $l_2$.

\item The proof of the second property is evident according to the definition of $N(c_i)$. For example if we replace the original labeling integer $p$ on $l_1$ by $p-k$, then $F(t)$ will be multiplied by $t^k$ and $G(t)$ will be multiplied by $t^{-k}$. The result follows.
\end{enumerate}
\end{proof}

The theorem above suggests us to define the linking polynomial $L_{K_1\cup K_2}(t)$ of a 2-component virtual link $L=K_1\cup K_2$ to be
\begin{center}
$L_{K_1\cup K_2}(t)=F(t)G(t)$.
\end{center}

The theorem below follows directly from Theorem 5.3.

\begin{theorem}
The linking polynomial $L_{K_1\cup K_2}(t)$ is a well defined virtual link invariant.
\end{theorem}

\textbf{Remark} In the definition of the linking polynomial we only consider the non-self-crossings of $L$, so that we can give a generalization of the linking number. It is natural to consider all the self-crossings with nontrivial indexes simultaneously, some similar polynomial invariant can be defined like the writhe polynomial. Of course this is just the combination of some ``knotted" information of each component and the ``intertwined" information of the link, hence it contains nothing new essentially.

From the definition of the linking polynomial one can find an apparent disadvantage of it. In fact there are virtual links with nontrivial linking number but trivial linking polynomial. For example for a link $L=K_1\cup K_2$, if at each crossing $c\in K_1\cap K_2$ we have $K_1$ lies over $K_2$, then obviously $G(t)=0$, hence one obtains $L_{K_1\cup K_2}(t)=F(t)G(t)=0$, while the linking number of $L$ may be nontrivial. In this case we can reconsider the pair of polynomials $F(t)$ and $G(t)$. It follows from Theorem 5.3 that $F(t)$ and $G(t)$ are both virtual link invariant up to multiplication by $t^{\pm k}$. From this perspective $F(t)$ and $G(t)$ are more powerful than the linking number and span, recall that
\begin{center}
$F(1)+G(1)=2lk(L)$, $|F(1)-G(1)|=span(L)$.
\end{center}
We can read more information about the minimal number of non-self-crossings from this pair of polynomials. Let $c_n(L)$ denote the minimal number of the non-self-crossings of a 2-component virtual link $L$, and $s(L)$ denote the sum of the absolute values of the coefficients of $F(t)$ and $G(t)$, i.e.
\begin{center}
$s=\sum\limits_{i}|a_i|+\sum\limits_{j}|b_j|$, if $F(t)=\sum\limits_{i}a_it^i$ and $G(t)=\sum\limits_{j}b_jt^j$.
\end{center}
The following proposition is evident, which is analogous to Proposition 4.3.
\begin{proposition}
Given a 2-component virtual link $L$, we have $s(L)\leq c_n(L)$.
\end{proposition}

The figure below gives an interesting and illuminating example.
\begin{center}
\includegraphics{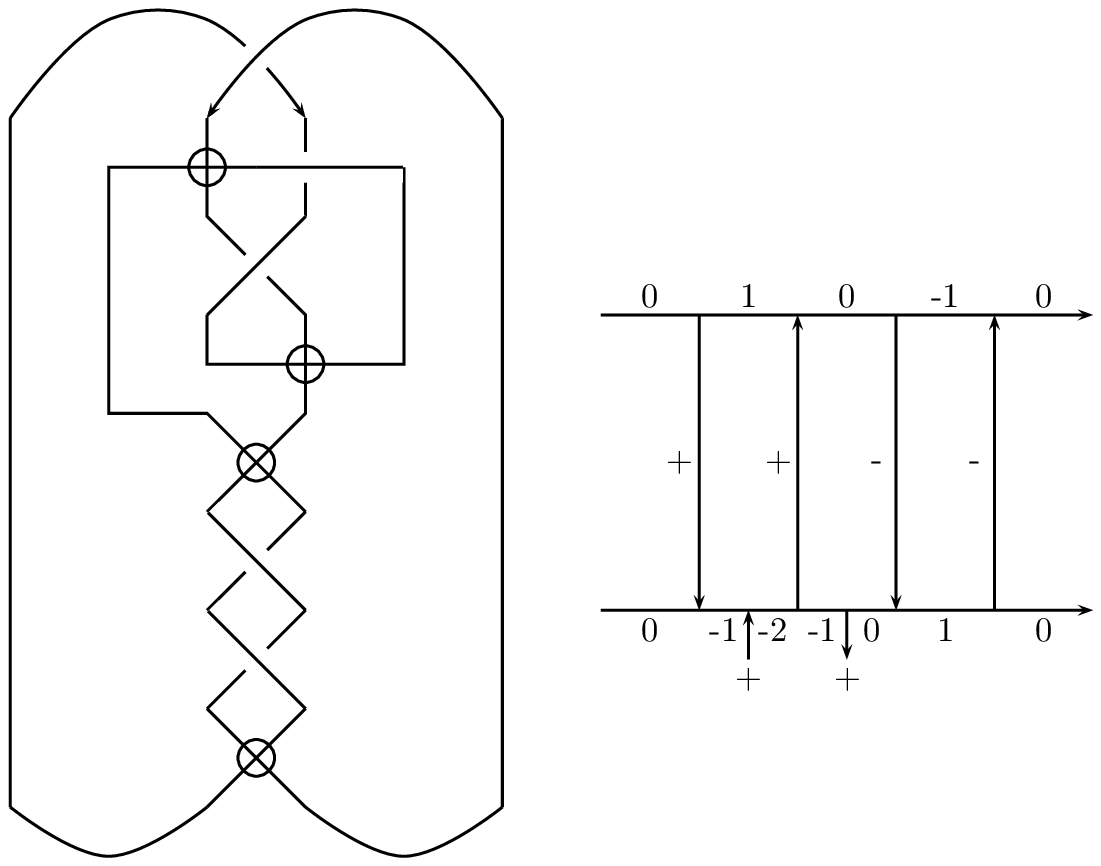}\\
Figure 35
\end{center}
Direct calculation shows that $L_{K_1\cup K_2}(t)=F(t)G(t)=(t^{-1}-t)(t^2-t^{-1})$. However the linking number and the span of $L$ are both trivial in this example. Hence the linking polynomial $L_{K_1\cup K_2}(t)$ $($or the pair $\{F(t), G(t)\})$ really contains more information than the linking number and the span. Moreover we can also conclude that the minimal number of non-self-crossings is exactly 4 in this case, according to Proposition 5.5.

Just like the writhe polynomial is trivial for classical knots, the linking polynomial of a classical link also can not offer any more information than the linking number. In fact we have the following proposition.
\begin{proposition}
Given a classical 2-component link $L=K_1\cup K_2$, we have $L_{K_1\cup K_2}(t)=(lk(L))^2$.
\end{proposition}
\begin{proof}
As we mentioned before, the linking number is preserved if one switches a self-crossing of $L$. In fact the linking polynomial also has the same property, namely if one makes a self-crossing change the linking polynomial is invariant. In order to see this we just need to consider the figure below, it is easy to find that the assigned integers on $l_1$ and $l_2$ are not changed under the switching.
\begin{center}
\includegraphics{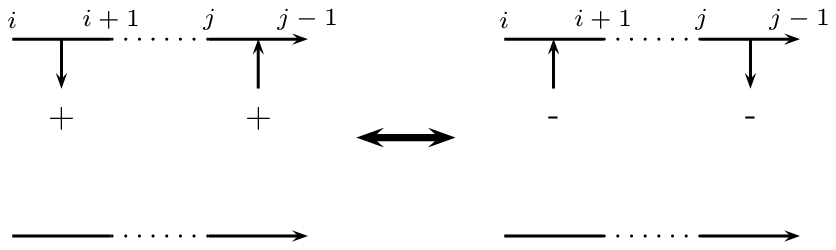}\\
Figure 36
\end{center}
Therefore given a classical link $L$ with linking number $n$, it suffices to consider another link $L'$ which is homotopic to $L$. Here we say a pair of links $L=K_1\cup K_2$ and $L'=K_1'\cup K_2'$ are \emph{homotopic} \cite{Mil1954} if there exists homotopies $f_i(t)$ $(i=1, 2)$ so that $f_i(0)=K_i, f_i(1)=K_i'$ and $f_1(t)\cap f_2(t)=\emptyset$ for each value of $t\in[0, 1]$. Obviously the homotopy classes of a 2-component link are completely determined by the linking number. Hence let us consider the closure of the 2-strand braid $(\sigma_1)^{2n}$, which also has linking number $n$. Direct calculation shows its linking polynomial is $n^2$, which implies the linking polynomial of $L$ is also $n^2$. This finishes the proof.
\end{proof}

We end this section with some simple properties of the linking polynomial. The symbol $\doteq$ means equal up to multiplication by $t^{\pm k}$. The proof is omitted.
\begin{proposition}\textcolor[rgb]{1.00,1.00,1.00}{c}

\begin{enumerate}
  \item Given a link $L$ and its inverse $\overline{L}$ $($the orientations of two components are both reversed$)$, we have $F_L(t)\doteq F_{\overline{L}}(t^{-1})$, $G_L(t)\doteq G_{\overline{L}}(t^{-1})$, and $L_{K_1\cup K_2}=L_{\overline{K_1}\cup \overline{K_2}}(t^{-1})$.
  \item Given a link $L$ and its mirror image $L^*$, we have $F_L(t)\doteq-G_{L^*}(t^{-1})$, $G_L(t)\doteq-F_{L^*}(t^{-1})$, and $L_{K_1\cup K_2}=L_{({K_1\cup K_2})^*}(t^{-1})$.
\end{enumerate}
\end{proposition}

\textbf{Acknowledge} The authors wish to thank the referees for their useful suggestions and comments. Some part of this paper was written during the first author's visit to Hiroshima University and Osaka City University, he would like to thank Dr. Ayaka Shimizu for her hospitality and the helpful conversations.


\begin{thebibliography}{99}
\bibitem{Che2012} Zhiyun Cheng. A polynomial invariant of virtual knots. arXiv: math.GT/1202.3850v1, 2012. To appear in Proceedings of the AMS
\bibitem{Dye2009} H. A. Dye, L. H. Kauffman. Virtual crossing number and the arrow polynomial. Journal of Knot Theory and Its Ramifications, 2009, 18: 1335-1357
\bibitem{Fol2012} Lena C. Folwaczny, Louis H. Kauffman. A linking number definition of the affine index polynomial and applications. arXiv:1211.1747v1, 2012
\bibitem{Gou2000} M. Goussarov, M. Polyak, O. Viro. Finite-type invariants of classical and virtual knots. Topology, 2000, 39: 1045-1068
\bibitem{Hen2010} Allison Henrich. A Sequence of Degree One Vassiliev Invariants for Virtual Knots. Journal of Knot Theory and Its Ramifications, 2010, 19: 461-487
\bibitem{IKL2012} Y. H. Im, S. Kim, D. S. Lee. The parity writhe polynomials for virtual knots and flat virtual knots. Journal of Knot Theory and Its Ramifications, 2012, 22: 1250133
\bibitem{YHI2010} Y. H. Im, K. Lee. A polynomial invariant of long virtual knots. Europ. J. Combinatorics, 2009, 30: 1289-1296
\bibitem{YHI2012} Y. H. Im, S. Y. Lee. A four-variable index polynomial invariant of long virtual knots. Journal of Knot Theory and Its Ramifications, 2012, 21: 1250083
\bibitem{Ish2010} Atsushi Ishii, Naoko Kamada, Seiichi Kamada. The Miyazawa polynomial for long virtual knots. Topology and Its Applications, 2010, 157: 290-297
\bibitem{Kau1999} Louis H. Kauffman. Virtual knot theory. Europ. J. Combinatorics, 1999, 20: 663--691
\bibitem{Kau2004} Louis H. Kauffman. A self-linking invariant of virtual knots. Fund. Math., 2004, 184: 135--158
\bibitem{Kau2005} Louis H. Kauffman. Knot diagrammatics. ``Handbook of knot theory", Elsevier B. V., 2005
\bibitem{Kau2011} Louis H. Kauffman. Introduction to Virtual Knot Theory. Journal of Knot Theory and Its Ramifications, 2012, 21: 1240007
\bibitem{Kau2012} Louis H. Kauffman. An affine index polynomial invariant of virtual knots. Journal of Knot Theory and Its Ramifications, 2013, 22: 1340007
\bibitem{Kup2003} Greg Kuperberg. What is a virtual link? Algebraic and Geomertric Topology, 2003, 3: 587-591
\bibitem{Man2004} Vassily Manturov. Knot theory. Chapman\&Hall/CRC, 2004
\bibitem{Man2004'} Vassily Manturov. Long virtual knots and their invariants. Journal of Knot Theory and Its Ramifications, 2004, 13: 1029-1039
\bibitem{Man2008} Vassily Manturov. Compact and long virtual knots. Trans. Moscow Math. Soc. 2008: 1-26
\bibitem{Man2010} Vassily Manturov. Parity in knot theory. Sbornik: Mathematics, 2010, 201:5 693-733
\bibitem{Man2012} Vassily Manturov, Denis Ilyutko. Virtual Knots-The State of the Art. World Scientific, 2012
\bibitem{Mil1954} John Milnor. Link groups. Annals of mathematics, 1954, 59(2): 177-195
\bibitem{Miy2008} Y. Miyazawa. A multivariable polynomial invariant for unoriented virtual knots and links. Journal of Knot Theory and Its Ramifications, 2008, 17: 1311-1326
\bibitem{Pol2010} Michael Polyak. Minimal generating sets of Reidemeister moves. Quantum Topology, 2010, 1: 399-411
\bibitem{Rol1976} D. Rolfsen. Knots and links. Publish or Perish, INC., Berkeley, CA., 1976
\bibitem{Shim2010} Ayaka Shimizu. The warping degree of a knot diagram. Journal of Knot Theory and Its Ramifications, 2010, 19: 849-857
\bibitem{Shim2011} Ayaka Shimizu. The warping polynomial of a knot diagram. Journal of Knot Theory and Its Ramifications, 2012, 21: 1250124
\end{thebibliography}
\end{document}